\documentclass[a4paper] {article}

\usepackage{color}
\usepackage{amsmath} 
\usepackage{amssymb}
\usepackage{amsthm}
\usepackage{graphicx}

\def\+{\oplus}

\newcommand{\R}{\mathbb{R}}
\newcommand{\N}{\mathbb{N}}

\newcommand{\T}{\mathbb{T}}

\newcommand{\cF}{{\mathcal F}}
\newcommand{\cG}{{\mathcal G}}

\newcommand{\cJ}{{\mathcal J}}
\newcommand{\cA}{{\mathcal A}}
\newcommand{\cB}{{\mathcal B}}

\newcommand{\cC}{{\mathcal C}}

\newcommand{\cM}{{\mathcal M}}
\newcommand{\cP}{{\mathcal P}}

\newcommand{\cK}{{\mathcal K}}

\DeclareMathOperator{\EE}{\mathbb{E}}
\DeclareMathOperator{\PP}{\mathbb{P}}

\newcommand{\diver}{{\rm{div}}}
\newcommand{\ds}{\displaystyle}
\def\squareforqed{\hbox{\rlap{$\sqcap$}$\sqcup$}}
\def\qed{\ifmmode\else\unskip\quad\fi\squareforqed}
\def\smartqed{\def\qed{\ifmmode\squareforqed\else{\unskip\nobreak\hfil
\penalty50\hskip1em\null\nobreak\hfil\squareforqed
\parfillskip=0pt\finalhyphendemerits=0\endgraf}\fi}}

\newtheorem{remark}{\textbf{Remark}}[section]

\newtheorem{lemma}{\textbf{Lemma}}[section]
\newtheorem{theorem}{\textbf{Theorem}}[section]

\newtheorem{proposition}{\textbf{Proposition}}[section]

\newtheorem{definition}{\textbf{Definition}}[section]

\numberwithin{equation}{section}
\title{Mean field type control with congestion}
\author{Yves Achdou \thanks { Univ. Paris Diderot, Sorbonne Paris Cit{\'e}, Laboratoire Jacques-Louis Lions, UMR 7598, UPMC, CNRS, F-75205 Paris, France.
 achdou@ljll.univ-paris-diderot.fr},
Mathieu Lauri{\`e}re \thanks { Univ. Paris Diderot, Sorbonne Paris Cit{\'e}, Laboratoire Jacques-Louis Lions, UMR 7598, UPMC, CNRS, F-75205 Paris, France.}
}

\begin{document}
\maketitle
\begin{abstract}
  We analyze some systems of partial differential equations arising in the theory of mean field type control with congestion effects.  We look for weak solutions. Our main result is the existence and uniqueness of suitably defined weak solutions, which are characterized as the optima of two optimal  control problems in duality. 
\end{abstract}
\section{Introduction}
In the recent years,  an important research activity has been devoted to the study of stochastic differential games  with a large number of players.
In their pioneering articles  \cite{MR2269875,MR2271747,MR2295621}, J-M. Lasry and P-L. Lions have introduced the notion of mean field games, 
which describe the   asymptotic behavior of  stochastic differential games (Nash equilibria) as the number $N$ of players
tends to infinity.  In  these models, it is assumed that the agents are all identical and that 
an individual agent can hardly influence the outcome of the game.  Moreover, each individual strategy is influenced by some averages of functions of 
the states of the other agents. In the limit when $N\to +\infty$, a given agent feels the presence of the other agents through the 
statistical distribution of the states of the other players. Since perturbations of a single agent's strategy does not influence the statistical distribution of the states, 
the latter acts as a parameter  in the control problem to be solved by each  agent.  
\\
Another kind of asymptotic regime is obtained by  assuming that all the agents use the same distributed feedback strategy
 and by passing  to the limit as $N\to \infty$ before optimizing the common feedback. Given a common feedback strategy, the asymptotics are 
given by the McKean-Vlasov theory, \cite{MR0221595,MR1108185} : the dynamics of a given agent is found by solving  a stochastic differential equation with coefficients depending on 
a mean field, namely the statistical distribution of the states, which may also affect the objective function. Since the feedback strategy is common to all agents, perturbations of the latter affect the mean field.  Then, having each player optimize its objective function amounts to solving a control problem
 driven by the McKean-Vlasov dynamics. The latter is named control of McKean-Vlasov dynamics by R. Carmona and F. Delarue \cite{MR3045029,MR3091726} and mean field type control by A. Bensoussan et al, \cite{MR3037035,MR3134900}.
\\
When the dynamics of the players are independent stochastic processes, both mean field games and control of  McKean-Vlasov dynamics 
naturally lead to a coupled system of partial differential equations, a forward Fokker-Planck equation
 and a backward Hamilton-Jacobi--Bellman equation.
For mean field games, the coupled system of partial differential equations has been studied by Lasry and Lions in  \cite{MR2269875,MR2271747,MR2295621}. Besides, many important aspects of the mathematical theory developed by  J-M. Lasry and P-L. Lions on MFG are not published in journals or books, but can be found in the videos of  the lectures of P-L. Lions at Coll{\`e}ge de France: see the web site of Coll{\`e}ge de France, \cite{PLL}. One can also see \cite{MR3195844} for a brief survey, and we mention \cite{porretta2014}, a very nice article on weak solutions of Fokker-Planck equations and of  MFG  systems of partial differential equations.
\\
 The analysis of the system of partial differential equations arising from mean field type control can be performed with rather similar arguments as for MFG, see \cite{YAML} for a work devoted to classical solutions. 

\bigskip

The class  of   MFG with congestion effects  was introduced and studied in  \cite{PLL} in 2011, see also \cite{MR3135339, YAML} for some numerical simulations, to model  situations in  which the cost of displacement of the agents increases in the regions where the density is large.
A striking fact is that in general,  MFG with congestion cannot be cast into an optimal control problem driven by a partial differential equation, in contrast with simpler cases. In the present paper, we aim at studying mean field type control with congestion, in a setting in which classical solutions of the system of partial differential equations seem difficult to obtain. %, to the best of our knowledge.
 But, in contrast with MFG, mean field type control can genuinely be seen as a problem of optimal control of a  partial differential equation.  This will allow us to use techniques from the calculus of variations. Inspired by the works of Cardaliaguet et al,  see \cite{MR3116016,cardaliaguet2014second}, we will introduce a pair of primal and dual optimization problems, leading  to a suitable weak formulation of the system of partial differential equations for which there exists a unique solution. Note that \cite{MR3116016} is devoted to some optimal transportation  problems (i.e. finding the geodesics 
for a class of distances between probability measures), whereas \cite{cardaliaguet2014second} deals with some special cases of MFG with possibly degenerate diffusions to which the above mentioned techniques from the calculus of variations can be applied.

\subsection{Model and assumptions}
\label{sec:model-assumptions}
This paper is devoted to the analysis of the second order system 
\begin{eqnarray}
\label{eq:1}
%\begin{array}[b]{l}
  \ds
\frac{\partial u} {\partial t} (t,x) + \nu \Delta u(t,x) +  H
(x, m(t,x), D u(t,x)) % \\
\ds + m(t,x)
 \frac{\partial  H} {\partial m}  
(x,m(t,x), D u(t, x)) 
%\end{array}
  = 0 ,
\\
\label{eq:2}
	\ds \frac{\partial m} {\partial t} (t,x)  - \nu \Delta  m(t,x) +
 \diver\Bigl( m(t,\cdot) \frac{\partial H} {\partial p} (\cdot, m(t,\cdot),D u(t,\cdot))\Bigr)(x) =0,
\end{eqnarray}
with the initial and terminal conditions 
\begin{equation}
\label{eq:3}
  m(0,x)=m_0(x)\quad \hbox{and}\quad u(T,x) = u_T(x).
\end{equation}
\paragraph{Assumptions} We now list the assumptions on the Hamiltonian $H$,  the initial and terminal conditions $m_0$ and $u_T$.
These conditions are supposed to hold in all what follows.
\begin{description}
\item[H1] The Hamiltonian $H:  \T^d\times (0,+\infty) \times \R^d \to \R $ is of the form
  \begin{equation}
    \label{eq:4}
    H(x,m,p)= -\frac {|p|^\beta}{m^{\alpha}} +\ell (x,m),
% \left\{
%       \begin{array}[c]{ll}
%         -\frac {|p|^\beta}{m^{\alpha}} +\ell (x,m), \quad & \hbox{ if } m>0,\\ 
%         \ell (x,0) &          \hbox{ if } m=0  \hbox{ and }  p=0,\\ 
%         -\infty  \quad & \hbox{ otherwise }
%       \end{array}
% \right.
  \end{equation}
  with $1<\beta \le 2$ and $0\le \alpha<1$, and where $\ell $ is continuous cost function that will be discussed below. It is clear that $H$ is concave with respect to $p$.
  Calling $\beta^*$ the conjugate exponent of $\beta$, i.e. $\beta^* =\beta /(\beta-1)$, it is useful to note that
  \begin{eqnarray}
\label{eq:5}
  H(x,m, p)=  \inf_{\xi \in \R^d }   \left( \xi\cdot p +L(x,m,\xi)  \right) ,\\
\label{eq:6}  L(x,m,\xi)=(\beta-1) \beta^{-\beta^*}  m^{\frac \alpha {\beta- 1}} |\xi|^{\beta^*} +\ell (x,m),
% \left\{ \begin{array}[c]{ll}
% (\beta-1) \beta^{-\beta^*}  m^{\frac \alpha {\beta- 1}} |\xi|^{\beta^*} +\ell (x,m), \quad & \hbox{if } m>0,\\
%    \ell (x,0) &          \hbox{ if } m=0,\\
%         -\infty  \quad & \hbox{ otherwise } 
% \end{array}\right.
  \end{eqnarray}
that $L$ is convex with respect to $\xi$, and that
\begin{equation}\label{eq:7}
  L(x,m,\xi)= \sup_{p\in \R^d} (-\xi\cdot p + H(x,m,p)).
\end{equation}
  \item[H2] (conditions on the cost $\ell$) The function $\ell: \T^d \times \R_+\to \R$ is continuous with respect to both variables and
continuously differentiable with respect to $m$ if $m>0$. We also assume that  $m\mapsto m\ell(x, m)$ is strictly convex, and that 
there exist $q>1$ and two positive constants $C_1$ and $C_2$ such that
\begin{eqnarray}
\label{eq:8}  \frac 1 {C_1 } m^{q-1}- C_1 \le \ell(x,m)\le C_1   m^{q-1}+ C_1,\\
\label{eq:9} \frac 1 {C_2 } m^{q-1}- C_2\le m \frac {\partial\ell }  {\partial m}  (x,m) \le C_2  m^{q-1}+ C_2.
\end{eqnarray}
Moreover, since we can always add a constant to $L$, we can assume that
\begin{equation}
  \label{eq:10}
\ell (x,m)\ge 0,\quad \forall x\in \T^d, \; \forall m\ge 0. 
\end{equation}
The convexity assumption on $m\mapsto m\ell(x, m)$ implies that $m\mapsto mH(x,m, p)$ is strictly convex with respect to $m$.\\
Moreover, we assume that there exists a constant $C_3\ge 0$ such that 
\begin{equation}
  \label{eq:11}
| \ell(x,m)-\ell(y,m)|\le C_3(1+ m^{q-1}) |x-y|.
%\left|  \frac {\partial\ell }  {\partial x}  (x,m)\right| + \left| m \frac {\partial^2\ell }  {\partial x\partial m}  (x,m)\right|       \le C (1+\ell(x,m)).
\end{equation}
\item[H3] We assume that $\beta \ge  q^*$.
\item[H4](initial and terminal conditions) We assume that $m_0$ is of class $\cC^1$ on $\T^d$, that $u_T$ is of class $\cC^2$ on $\T^d$ and that $m_0> 0$ and $\int_{\T^d} m_0(x)dx=1$.
\item[H5]   $\nu$ is a positive number.
\end{description}

\begin{remark}
  \label{sec:assumptions}
Note that Assumption [H5] can be relaxed: all what follows can be generalized to degenerate diffusions, i.e. to the following system of PDEs:
\begin{eqnarray}
\label{eq:1bis}
%\begin{array}[b]{l}
  \ds
\frac{\partial u} {\partial t} (t,x) + A_{i,j}(x) \frac{\partial^2 u} {\partial x_i\partial x_j}  (t,x) +  H
(x, m(t,x), D u(t,x)) % \\
\ds + m(t,x)
 \frac{\partial  H} {\partial m}  
(x,m(t,x), D u(t, x)) 
%\end{array}
  = 0 ,
\\
\label{eq:2bis}
	\ds \frac{\partial m} {\partial t} (t,x)  - \nu   \frac{\partial^2 } {\partial x_i\partial x_j}   ( A_{i,j}(\cdot) m) (t,x) +
 \diver\Bigl( m(t,\cdot) \frac{\partial H} {\partial p} (\cdot, m(t,\cdot),D u(t,\cdot))\Bigr)(x) =0,
\end{eqnarray}
 where $(A_{i,j})_{1\le i,j\le d}(x) =\frac 1 2 \Sigma(x)\Sigma^T (x)$ and $\Sigma$ is a Lipschitz continous map from $\T^d$ to $\R^{d\times D}$ with $D$ possibly smaller than $d$. The necessary modifications can easily be found in \cite{cardaliaguet2014second}.
\end{remark}

\subsection{A heuristic justification of (\ref{eq:1})-(\ref{eq:3})}
Consider a probability space $(\Omega, \cA, \cP)$ and a filtration $\cF^t$ generated by a $d$-dimensional standard Wiener process $(W_t)$ and
the stochastic process $(X_t)_{t\in [0,T]}$   in $\R^d$ adapted to $\cF^t$ which solves the stochastic differential equation
\begin{equation}\label{eq:12}
  d X_t = \xi_t \;d t + \sqrt{2\nu} \;d W_t \qquad \forall t\in [0,T],
\end{equation}
given the initial state $X_0$ which is a random variable $\cF^0$-measurable whose probability density is  $m_0$.
In (\ref{eq:12}), $\xi_t$ is the control,  which we take to be
\begin{equation}\label{eq:13}
\xi_t=v(t,X_t), \end{equation}
where $v(t,\cdot)$ is a continuous function on $\T^d$.
As explained in \cite{MR3134900}, page 13, if the feedback function $v$ is smooth enough, then the probability distribution $m_{t}$ of $X_t$ has a density with respect to the Lebesgue measure,  $m_{v}(t,\cdot)\in \PP\cap L^1(\T^d)$ for all $t$,
and   $m_v$ is solution of the Fokker-Planck equation
 \begin{equation}
\label{eq:14}
\frac{\partial m_v} {\partial t} (t,x)  - \nu \Delta m_v(t,x) +
\diver \Big( m_v(t,\cdot)  v(t,\cdot)  \Big) (x)=0,\;\; 
 \end{equation}
for $t\in (0,T]$ and $ x\in \T^d$,
with the initial condition
\begin{equation}
  \label{eq:15}
m_v(0,x)= m_0(x),\quad x\in \T^d.
\end{equation}
 We define the objective function
\begin{equation}
  \label{eq:16}
  \begin{split}
    \cJ(v) &= \EE\left[ \int_0^T L( X_t,m_v(t, X_t), \xi_t) dt  + u_T(X_T)\right]\\
&=\ds \int_{[0,T]\times\T^d} L( x, m_v(t,x),v(t,x)) m_v(t,x) dx dt + \int_{\T^d} u_T(x) m_v(T,x) dx.    
  \end{split}
\end{equation}
The goal is to minimize $\cJ(v)$ subject to (\ref{eq:14}) and (\ref{eq:15}). Following A. Bensoussan, J. Frehse and P. Yam  in  \cite{MR3134900},   it can be seen   that if there exists a smooth feedback function $v^* $ achieving
$\cJ( v^*)= \min \cJ(v)$  and such that $ m_{v^*}>0$    then
\begin{displaymath}
  v^*(t,x)={\rm{argmin}}_{v}  \Bigl( L(x, m_{v^*}(t,x),v)+ \nabla u(t,x)\cdot v(t,x)\Bigr)
\end{displaymath}
and $(m_{v^*},u)$ solve~(\ref{eq:1}), (\ref{eq:2}) and (\ref{eq:3}). 
% The formal argument is  fully justified if one can show that (\ref{eq:1}), (\ref{eq:2}) and (\ref{eq:3}) has a unique and  smooth enough solution and that the density does not vanish. 
The issue with the latter argument
is that we do not know how to guarantee a priori that 
 $m$ will not vanish in some region of $(0,T)\times \T^d$. 
Hereafter, we propose a theory of weak solutions of (\ref{eq:1})-(\ref{eq:3}), in order to cope with the cases when $m$ may vanish. 
\begin{remark}
  \label{sec:heur-just-refeq:1}
Note that the system of partial differential equations that arises in a mean field game is
\begin{displaymath}
  \ds
\frac{\partial u} {\partial t} (t,x) + \nu \Delta u(t,x) +  H
(x, m(t,x), D u(t,x)) 
  = 0 ,
\end{displaymath}
with (\ref{eq:2}) and (\ref{eq:3}). To the best of our knowledge, for  such a system with the Hamiltonian given in (\ref{eq:4}),
 the existence of a solution is an open problem except in the stationary case with $\beta=2$ and $\alpha=1$, see \cite{gomes2014existence}; 
in the latter case, a very special trick can be used. Besides, the theory of weak solutions proposed below does not apply to MFG, because as explained above,
 MFG with congestion cannot be seen as an optimal control problem driven by a partial differential equation. 
\end{remark}

\section{Two optimization problems}
\label{sec:two-optim-probl}
The first optimization is described as follows: consider the set $\cK_0$:
\begin{displaymath}
  \cK_0=\left\{ \phi \in \cC^2 ([0,T]\times \T^d): \phi(T,\cdot)= u_T \right\}
\end{displaymath}
and the functional $\cA$ on $\cK_0$:
\begin{equation}
  \label{eq:17}
\cA(\phi)=\inf_{
  \begin{array}[c]{l}
m\in L^1((0,T)\times \T^d)\\ m\ge 0    
  \end{array}
}  \cA(\phi,m)
\end{equation}
where 
\begin{equation}
  \label{eq:18}
 \cA(\phi,m)=
  \begin{array}[t]{l}
\ds \int_0^T \int_{\T^d}    m(t,x)\left(\frac{\partial \phi} {\partial t} (t,x)  + \nu \Delta \phi(t,x)  +  H
(x, m(t,x), D \phi(t,x) )          \right)    dx dt    \\ \ds +\int_{\T^d}    m_0(x)\phi(0,x) dx
  \end{array}\end{equation}
with the convention that if $m=0$ then $mH(x,m,p)=0$.
Then the first problem consists of maximizing 
\begin{equation}
  \label{eq:19}
  \sup_{\phi\in \cK_0} \cA(\phi).
\end{equation}
For the second optimization problem, we consider the set  $\cK_1$:
\begin{equation}
  \label{eq:20}
  \cK_1=\left\{
    \begin{array}[c]{l}
(m,z) \in L^1 ((0,T)\times \T^d) \times L^1 ((0,T)\times \T^d;\R^d)  : \\ m\ge 0 \hbox{ a.e}
\\
	\ds \frac{\partial m} {\partial t}   - \nu \Delta  m +  \diver   z  =0 ,\\
 m(0,\cdot)= m_0 
    \end{array}
\right\}
\end{equation}
where the boundary value problem is satisfied in the sense of distributions.
We also define
\begin{equation}
\label{eq:21}
\widetilde L (x,m, z)= \left\{
  \begin{array}[c]{rl}
m L(x, m, \frac z m ) \quad &\hbox{ if } m>0\\
0    \quad &\hbox{ if } (m,z)=(0,0)\\
+\infty \quad &\hbox{ otherwise } 
  \end{array}
    \right.   .
\end{equation}
Note that $(m,z) \mapsto \widetilde L(x,m,z)$ is LSC on $\R\times \R^d$.
Using (\ref{eq:6}), (H2) and the results of \cite{YAML} paragraph 3.2, it can be proved that  
 $(m,z) \mapsto \widetilde L(x,m,z)$ is convex on $\R\times \R^d$, because $0<\alpha<1$. It can also be checked that
\begin{equation}\label{eq:22}
  \widetilde L(x,m,z)=\left\{
    \begin{array}[c]{ll}
\ds \sup_{p\in \R^d} (-z\cdot p + mH(x,m,p))\quad \quad &\hbox{if } m>0 \hbox{ or } (m,z)=(0,0)      \\
+\infty &\hbox{otherwise.}
    \end{array}\right.
\end{equation}
Since $L$ is bounded from below, $\ds \int_0^T \int_{\T^d} \widetilde L (x,m(t,x), z(t,x)) dx dt $ is well defined in $\R \cup \{+\infty\}$ for all $(m, z) \in \cK_1$.
 We are interested in minimizing 
 \begin{equation}
   \label{eq:23}
  \inf_{(m,z)\in \cK_1} \cB(m,z),
 \end{equation}
where if $\ds \int_0^T \int_{\T^d} \widetilde L (x,m(t,x), z(t,x)) dx dt<+\infty$,
\begin{equation}
  \label{eq:24}
\ds \cB(m, z)=  \int_0^T \int_{\T^d} \widetilde L (x,m(t,x), z(t,x)) dx dt  + \int_{\T^d}     m(T,x) u_T(x) dx,
\end{equation}
 and if not,
\begin{equation}
  \label{eq:25}
\cB(m, z)= +\infty.
\end{equation}
\\
 To give a meaning to the second integral in (\ref{eq:24}), we define $w(t,x)= \frac {z(t,x)}{m(t,x)}$ if $m(t,x)>0$ and 
$w(t,x)=0$ otherwise. From (\ref{eq:6}) and (\ref{eq:8}), we see that   $\int_0^T \int_{\T^d} \widetilde L (x,m(t,x), z(t,x)) dx dt <+\infty$ implies that
 $m ^{1+\frac \alpha {\beta-1}} |w|^{\beta ^*}\in L^1( (0,T)\times \T^d)$,   which implies that 
 $m  |w|^{\frac \beta  {\beta -1 +\alpha}}\in L^1( (0,T)\times \T^d)$. 
 %As a consequence $m |w|^{\beta ^*}\in L^1( (0,T)\times \T^d)$. 
In that case, the boundary value problem in (\ref{eq:20}) can be rewritten as follows:
\begin{equation}
  \label{eq:26}
	\ds \frac{\partial m} {\partial t}   - \nu \Delta  m +  \diver   (mw)  =0 ,\quad\quad m(0,\cdot)= m_0,
\end{equation}
and we can use the following Lemma which can be found in \cite{cardaliaguet2014second}:
\begin{lemma}
\label{sec:two-optim-probl-1}
If $(m,w)\in \cK_1$ is such that $\int_0^T \int_{\T^d} \widetilde L (x,m(t,x), z(t,x)) dx dt <+\infty$, then
  the map $t\mapsto m(t)$  for $t\in (0,T)$ and $t\mapsto m_0$ for $t<0$  is H{\"o}lder continuous a.e. for the weak * topology of $\cP(\T^d)$.
% with an exponent $1/2$.
\end{lemma}
\begin{remark}
  \label{sec:two-optim-probl-3}
Following the proof of lemma 3.1 in \cite{cardaliaguet2014second}, we see that
the H{\"o}lder exponent in Lemma \ref{sec:two-optim-probl-1} is greater than or equal to $\min \left( \frac 1 2 , \frac {1-\alpha} \beta\right)$.
\end{remark}
This lemma implies that the measure $m(t)$ is defined for all $t$, so the second integral in (\ref{eq:24}) has a meaning.
\begin{lemma}
  \label{sec:two-optim-probl-2}
  \begin{equation}
    \label{eq:27}
\sup _{\phi\in \cK_0} \cA(\phi)= \min_{(m,z)\in \cK_1} \cB(m,z).
  \end{equation}
Moreover the latter minimum is achieved by a unique $(m^*, z^*)\in \cK_1$, and $m^*\in L^q((0,T)\times \T^d)$.
\end{lemma}
\begin{proof}
  Let us reformulate the optimization problem (\ref{eq:19}): take $E_0=\cC^2([0,T]\times \T^d)$ and $E_1=\cC^0([0,T]\times \T^d)\times \cC^0([0,T]\times \T^d;\R^d)$.  We define the functional $\cF$ on $E_0$:
  \begin{equation*}
    \cF(\phi)= \chi_T(\phi)- \int_{\T^d} m_0(x)\phi(0,x)dx
  \end{equation*}
where $\chi_T(\phi)=0$ if $\phi|_{t=T}=u_T$ and $\chi_T(\phi)=+\infty$ otherwise. Let us also define the linear operator $\Lambda: E_0\to E_1$ by
\begin{displaymath}
  \Lambda(\phi)= \left(\frac{\partial \phi} {\partial t}   + \nu \Delta \phi, D\phi \right).
\end{displaymath}
For $(a,b)\in E_1$,  let $\cG(a, b)$ be defined by
\begin{equation}
  \label{eq:28}
\cG(a, b)= -\inf_{ \begin{array}[c]{l}
m\in L^1((0,T)\times \T^d)\\ m\ge 0    
  \end{array}} \int_0^T \int_{\T^d}    m(t,x) \left( a(t,x)+  H
(x, m(t,x), b(t,x) ) \right)    dx dt   .
\end{equation}
Note that the infimum with respect to $m$ is in fact a minimum:  indeed,  
 from (\ref{eq:4}) and (\ref{eq:8}), we see that 
$m\mapsto \ds \int_0^T \int_{\T^d}    m(t,x) \left( a(t,x)+  H(x, m(t,x), b(t,x) ) \right)    dx dt $ is convex, coercive and continuous in the set $\left\{ m\in L^q((0,T)\times \T^d):\; m\ge 0 \right \}$.  Hence, this map  is lower semi-continuous for the weak convergence in  $\{ m\in L^q((0,T)\times \T^d), m\ge 0 \}$.
On the other hand, since a  minimizing sequence  $(m_n)_{n\in \N}$ for (\ref{eq:28}) is bounded in $L^q((0,T)\times \T^d)$,  we can extract a subsequence which converges weakly in  $L^q((0,T)\times \T^d)$ to a nonnegative function.  The weak limit achieves the   minimum in (\ref{eq:28}). We now aim at characterizing the optimal $m$.
\\
Let us first characterize
\begin{equation}
  \label{eq:29}
K(x,\gamma, p)= \min_{\mu\ge 0} \left( \mu \gamma +\mu  H(x, \mu, p) \right),
\end{equation}
 which is nonpositive and concave with respect to $(\gamma, p)$;
 since $\mu\mapsto \mu H(x,\mu,p)$ is $\cC^1$, strictly convex on $\R_+$ 
and tends to $+\infty$ as $\mu\to +\infty$,  we see
that for any $x\in \T^d$,
 if $p\not =0$ and  $\gamma\in \R$, or if $p=0$ and $\gamma+\ell(x, 0) \le 0$,  
then there exists a unique $\mu=\psi(x, \gamma, p)\ge 0$ such that 
\begin{displaymath}
\gamma+ H(x,\mu, p)+ \mu H_m(x,\mu,p)=0.
\end{displaymath}
Note that if $p=0$ and $\gamma+\ell(x, 0) < 0$,
  then $\mu=\psi(x, \gamma, 0)>0$ is characterized by  $\gamma+\ell(x, \mu)+\mu \ell_m(x,\mu) = 0$.  
We extend $\psi$ by $0$ in the set $\{(\gamma,0):  \gamma+\ell(x, 0) \ge 0\}$. Therefore, 
\begin{equation}
  \label{eq:30}
  K(x,\gamma,p)=      \psi(x, \gamma, p) \gamma+  \psi(x, \gamma, p) H(x, \psi(x, \gamma, p), p),
\end{equation}
with the convention that $mH(x,m,p)=0$ if $m=0$.
\\
We claim that the map $ (\gamma, p)\mapsto \psi(x, \gamma, p)$ is continuous in 
$ \R\times \R^d$. Indeed,
\begin{enumerate}
\item the continuity of $(\gamma,m, p)\mapsto \gamma+ H(x,m, p)+ \mu H_m(x,m,p) $ and the fact that this map is strictly increasing w.r.t. $m$ in $(0,+\infty)$
implies that $(\gamma, p)\mapsto \psi(x, \gamma, p)$ is continuous in 
$ \R\times \R^d\backslash\{0\}$.
\item Similarly, the continuity of $\gamma\mapsto \psi(x, \gamma, 0)$ stems from the  continuity of  the map
 $(\gamma,m)\mapsto  \gamma+\ell(x, m)+m \ell_m(x,m) $  and its strictly increasing character w.r.t. $m$ in $(0,+\infty)$.
\item 
Let us prove  that if    $(\tilde \gamma ,\tilde p)$ tends to $(\gamma,0)$ with  $ \tilde p\not = 0$, then 
\begin{equation}
  \label{eq:66}
\lim_{( \tilde \gamma,\tilde p) \to ( \gamma,0) } \psi(x,  \tilde \gamma,\tilde p)= \psi(x,  \gamma,0).
\end{equation}
\begin{enumerate}
\item
If $ \gamma+\ell(x, 0) < 0$, then $\psi(x, \gamma, 0)>0$ and we get (\ref{eq:66}) from the same argument as in point 1.
\item Consider the case $ \gamma+\ell(x, 0) > 0$.  Suppose that
$(\tilde \gamma ,\tilde p)$ tends to $(\gamma,0)$ with  $ \tilde p\not = 0$,  
 and set $\tilde \mu =\psi(x,  \tilde \gamma, \tilde p)>0$.
We see that
\begin{displaymath}
  \begin{split}
    (1-\alpha) \tilde \mu^{-\alpha}   |\tilde p|^{\beta} = 
\tilde \gamma +   \frac d {dm}  (m\mapsto m\ell(x,m))(\tilde \mu)
    >  \tilde \gamma + \ell(x,0)    \to \gamma +\ell(x,0)  .
  \end{split}
\end{displaymath}
This implies that $\tilde \mu^\alpha  |\tilde p|^{-\beta}$ is bounded as  $(\tilde \gamma, \tilde p) \to  (\gamma,0)$, hence (\ref{eq:66}). 
\item Finally, we consider the case when $ \gamma=-\ell(x, 0) \le 0$ and $(\tilde \gamma, \tilde p) \to ( \gamma,0)$ with $\tilde p\not =0$;
 let us assume that for a subsequence, $\tilde \mu$ is bounded away from $0$: passing to the limit in the identity
\begin{displaymath}
  \left(   \frac d {dm}  (m\mapsto m\ell(x,m))(\tilde \mu)-\ell (x,0) \right) -(1-\alpha) \tilde \mu^{-\alpha} |\tilde p|^{\beta}=\gamma-\tilde \gamma,
\end{displaymath}
we obtain that 
\begin{displaymath}
  \lim_{(\tilde \gamma, \tilde p) \to ( \gamma,0)}   \frac d {dm}  (m\mapsto m\ell(x,m))(\tilde \mu)=\ell (x,0), 
\end{displaymath}
which can happen only if $\tilde \mu\to 0$ and we reach a contradiction.
 Hence, (\ref{eq:66}) holds.
\end{enumerate}
\end{enumerate}
We have proved the continuity of $\psi$ with respect to $(\gamma, p)$. The continuity of $\psi$ with respect to $x$ follows from similar arguments, 
using the  regularity assumptions on $\ell $.  Therefore,  $\psi$ is  continuous in the set $ \T^d\times \R\times \R^d$.

It is also useful to notice that 
\begin{displaymath}
  \begin{split}
   K(x,\gamma,p)&= \inf_{\mu \ge 0}\left(\mu(\gamma+ H(x,\mu,p) \right)    \\
&= \inf_{\mu \ge 0}\left(\mu\gamma+\mu \inf_{\xi}\left( \xi\cdot p + L(x,\mu,\xi) \right)\right)\\
 &=\inf_{(\mu,z) \in \R\times \R^d } \left(\mu\gamma+  z\cdot p + \widetilde 
L(x,\mu,z) \right)
  \end{split}
\end{displaymath}
and, from the Fenchel-Moreau theorem, see e.g.  see \cite{MR1451876}, 
 that
\begin{equation}
  \label{eq:31}
  \widetilde 
L(x,\mu,z)=\sup_{(\gamma,p)\in \R\times \R^d }\left(- \mu\gamma-  z\cdot p+K(x,\gamma,p)  \right).
\end{equation}
\begin{remark}
  \label{sec:two-optim-probl-4}
Note that for all $0\not = p\in \R^d$ and $\gamma\in \R$, 
 the map $\mu\mapsto  \mu \left(\gamma+ H(x,\mu,p) \right)$ is strictly 
decreasing in some interval $[0, \bar \mu]$ where $\bar \mu >0$ depends on $p$ and $\gamma$,
 and that its derivative tends to $-\infty$ as $\mu\to 0^+$. Hence, 
if $\mu^* =0$ achieves the minimum of  $\mu\mapsto  \mu \left (\gamma+ H(x,\mu,p) \right)$, then $p$ must be $0$. 
Similarly, $\gamma$ must be such that $\gamma+\ell (x,0)\ge 0$.
\end{remark}

With (\ref{eq:30}), the optimality conditions for (\ref{eq:28})   yield that
\begin{equation}
  \label{eq:32}
  \cG(a,b)= - \int_0^T \int_{\T^d}  K(x,a(t,x),b(t,x))   dx dt.
\end{equation}
From (\ref{eq:30}) and the continuity of $\psi$, we see that $\cG$ is continuous on $E_1$. 
We observe that 
\begin{equation}
  \label{eq:33}
  \sup_{\phi\in \cK_0} \cA(\phi)=   - \inf_{\phi\in E_0} \left( \cF(\phi) + \cG (\Lambda(\phi)) \right).
\end{equation}
By choosing $\phi_0(t,\cdot)= u_T$, we see that $\cF(\phi_0)<+\infty$, $\cG(\Lambda (\phi_0))<+\infty$ and that 
$\cG\circ \Lambda$ is continuous at $\phi_0$.
We can thus apply Fenchel-Rockafellar duality theorem, see \cite{MR1451876}: 
\begin{equation}
  \label{eq:34}
-  \inf_{\phi\in E_0} \left( \cF(\phi) + \cG (\Lambda(\phi)) \right)=\min_{(m,z)\in E_1^* } \left(      \cF^* (\Lambda^* (m,z)) + \cG^* (-m,-z) \right)
\end{equation}
where $E_1^*$ is the topological dual of $E_1$ i.e. the set of Radon measures $(m,z)$ on $(0,T)\times \T^d$ with values in $\R\times \R^d$. If $E_0^*$ is the dual space of $E_0$, the operator $\Lambda^*: E_1^* \to E_0^*$ is the adjoint of $\Lambda$. The maps $\cF^*$ and $\cG^*$ are the Legendre-Fenchel conjugates of $\cF$ and $\cG$. Following \cite{cardaliaguet2014second}, we check that 
\begin{displaymath}
    \cF^* (\Lambda^* (m,z))=\left\{
      \begin{array}[c]{ll}
       \ds  \int_{\T^d }  u_T(x) dm(T,x) \quad & \hbox{if }
       \left\{ \begin{array}[c]{l}
\ds \frac{\partial m} {\partial t}   - \nu \Delta  m +  \diver   z  =0 ,\\
m(0,\cdot)= m_0           
        \end{array}\right.
\\
+\infty \quad & \hbox{otherwise. }
      \end{array}
\right.
\end{displaymath}
where the boundary value problem is understood in the sense of distributions.
\\
On the other hand, from Rockafellar, \cite{MR0310612} Theorem 5, and
(\ref{eq:31}), see also \cite{MR3116016}, we see that 
\begin{displaymath}
  \cG^* (-m,-z)=\int_0^T \int_{\T^d} \widetilde L (x,m^{\rm ac} (t,x), z^{\rm ac}(t,x)) dxdt + \int_0^T \int_{\T^d} \widetilde L_{\infty}\left(x, \frac {dm^{\rm sing}}{d\theta},\frac {dz^{\rm sing}}{d\theta} \right)  d\theta,
\end{displaymath}
where $(m^{\rm ac}, z^{\rm ac})$ and $(m^{\rm sing},z^{\rm sing})$ respectively denote the absolutely continuous and singular parts of $(m,z)$, $\theta$ is any measure with respect to which $(m^{\rm sing},z^{\rm sing})$ is absolutely continuous, (for instance $m^{\rm sing}+|z^{\rm sing}|$, and $\widetilde L_\infty(x,\cdot)$ is the recession function of $\widetilde L(x,\cdot)$, i.e.
\begin{displaymath}
  \widetilde L_\infty(x,m,z)=\sup_{\lambda>0} \frac 1 \lambda \widetilde L (x,\lambda m, \lambda z)=\left\{
    \begin{array}[c]{ll}
      0\quad&\hbox{if } (m,z)=(0,0),\\
+\infty \quad &\hbox{otherwise}.
    \end{array}
\right.
\end{displaymath}
Therefore, 
\begin{equation}
  \label{eq:35}
  \cG^* (-m,-z)=\left\{ 
    \begin{array}[c]{ll}
\ds \int_0^T \int_{\T^d} \widetilde L (x,m(t,x), z(t,x)) dxdt      \quad&  \hbox{if } (m,z) \in L^1 ((0,T)\times \T^d) \times L^1 ((0,T)\times \T^d;\R^d)\\
+\infty  \quad&  \hbox{otherwise.}
    \end{array}
 \right.
\end{equation}
Hence,
\begin{displaymath}
  \min_{(m,z)\in E_1^* } \left(      \cF^* (\Lambda^* (m,z)) + \cG^* (-m,-z) \right)= \min_{(m,z)\in \cK_1} \cB(m,z).
\end{displaymath}
and we obtain the desired result from (\ref{eq:33}) and (\ref{eq:34}).\\
Using (\ref{eq:6}), the strict convexity of $m\mapsto m\ell (x,m)$ assumed in (H2), the convexity of the map $(m>0,z)\mapsto m^{1+\frac \alpha {\beta- 1}} |\frac z m|^{\beta^*} $ (see \cite{YAML} paragraph 3.2),  and the convexity of $\cK_1$, we 
obtain the uniqueness of $m^*$ such that $(m^* , z^*)\in \cK_1$ achieves a 
minimum of $\cB$ for some $z^*$. Moreover, from the strict convexity of 
$(m,z)\mapsto m ^{1+\frac \alpha {\beta- 1}} |\frac z m|^{\beta^*} $ for $m>0$, we deduce that $ \frac {z^*}{m^* }$ is unique in $\{  (t,x):  m^* (t,x)>0 \}$. Since $z^*=0$ in $\{  (t,x):  m^* (t,x)=0 \}$, the uniqueness of $z^*$ follows.
It is clear from (H2) that $m^* \in L^q((0,T)\times \T^d)$. 
% and that 
% \begin{displaymath}
%   \int_0^T \int_{T^d} |z^*|^{ \frac {q \beta^*}{q +\beta^*}}
% \le   \left \int_0^T \int_{T^d} m^q\right) ^{}
%   \int_0^T \int_{T^d}
%   \int_0^T \int_{T^d}
%   \int_0^T \int_{T^d}
% \end{displaymath}
\end{proof}

\section{A priori estimates for a maximizing sequence of (\ref{eq:17}) }
\label{sec:priori-estim-maxim}
Let $M\in \R$ be the optimal value in (\ref{eq:17}).  Take a maximizing sequence $(\phi_n)_{n\in N}$ for (\ref{eq:17}). For some $\epsilon>0$, it can be chosen in such a way 
that
\begin{displaymath}
  M-\epsilon<\cA(\phi_n)\le M.
\end{displaymath}
From the definition of $\cA$, we see that $\cA(\phi_n)\le \cA(\phi_n, m_0)$. Hence, 
\begin{displaymath}
  \begin{array}[c]{rcl}
    M-\epsilon&\le &   \ds \int_0^T \int_{\T^d}    m_0(x)\left(\frac{\partial \phi_n} {\partial t} (t,x)  + \nu \Delta \phi_n(t,x)  +  H
(x, m_0(x), D \phi_n(t,x) )          \right)    dx dt    \\ & & \ds +\int_{\T^d}    m_0(x)\phi_n(0,x) dx\\
&=& 
 \ds \int_0^T \int_{\T^d}    \left( -\nu   Dm_0(x)\cdot D \phi_n(t,x)  +  m_0(x) H
(x, m_0(x), D \phi_n(t,x) )    \right)    dx dt +\int_{\T^d}    m_0(x)u_T(x) dx.
  \end{array}
\end{displaymath}
which implies that $\| D\phi_n \|_{L^\beta ((0,T)\times \T^d)}$ is bounded uniformly w.r.t. $n$. \\
Let $m_n$   achieve
$\cA(\phi_n)=\cA(\phi_n, m_n)$. Recall that $m_n$ is unique. 
% and that 
% $ m_n (t,x)=\psi\left(\frac{\partial \phi_n} {\partial t}  (t,x)  + \nu \Delta \phi_n (t,x) , D\phi_n  (t,x)\right)$ 
% if $\frac{\partial \phi_n} {\partial t}  (t,x)  + \nu \Delta \phi_n (t,x) + H(x, 0, D\phi_n  (t,x))\le 0$, and $   m_n (t,x)=0$ otherwise. 
\\
  The optimality conditions for $m_n$ in (\ref{eq:17}) are 
\begin{eqnarray}
  \label{eq:36}
\frac{\partial \phi_n} {\partial t} (t,x)  + \nu \Delta \phi_n(t,x)  +  H
(x, m_n(t,x), D \phi_n(t,x) )     + m_n(t,x) H_m((x, m_n(t,x), D \phi_n(t,x) ) \ge 0 \quad \hbox{a.e.},\\
\label{eq:37}
\ds 0=  \left \{
  \begin{array}[c]{l} \ds
\int_0^T \int_{\T^d}    m_n(t,x)\left(\frac{\partial \phi_n} {\partial t} (t,x) + \nu \Delta \phi_n(t,x)  \right)   dx dt\\
+  \ds  \int_0^T \int_{\T^d}    m_n(t,x)\left(H
(x, m_n(t,x), D \phi_n(t,x) )  + m_n(t,x) H_m(x, m_n(t,x), D \phi_n(t,x) ) \right) dx dt.
  \end{array}\right.
\end{eqnarray}
From (\ref{eq:37}), we deduce that
\begin{displaymath}
  \cA(\phi_n)=  -\int_0^T \int_{\T^d}     m_n^2(t,x) H_m((x, m_n(t,x), D \phi_n(t,x) )dx dt  +\int_{\T^d}    m_0(x)\phi_n(0,x) dx.
\end{displaymath}
From (H1) and (H2),
\begin{displaymath}
m^2H_m(x,m,p) \ge  \frac 1 {C_2} m^q -C_2 m  + m^{1-\alpha} |p|^\beta.  
\end{displaymath}
Hence,
\begin{equation}
  \label{eq:38}
  \begin{split}
&\ds   \int_0^T \int_{\T^d}  \left(  \frac 1 {C_2} m_n^q(t,x) -C_2 m_n(t,x)  + m_n^{1-\alpha}(t,x) |D \phi_n(t,x)|^\beta  \right)  
dx dt  -\int_{\T^d}    m_0(x)\phi_n(0,x) dx\\ & \le -M +\epsilon    .
  \end{split}
\end{equation}
On the other hand, (\ref{eq:36}), (H1) and (H2) imply that for some constant $C_4>0$,
\begin{equation}
  \label{eq:39}
\frac{\partial \phi_n} {\partial t} (t,x)  + \nu \Delta \phi_n(t,x)  + C_4 m_n^{q-1}(t,x) + C_4\ge 0,\quad \hbox{a.e.}
\end{equation}
Multiplying (\ref{eq:39}) by $(\phi_n^+)^{q*-1} e^{-\lambda (T-t)} $ for $\lambda $ large enough, integrating in $(\tau,T)\times \T^d$, we obtain that
\begin{equation}
  \label{eq:40}
\|\phi_n^+\|_{L^\infty(0,T;L^{q^*}(\T^d) )}
%+  \|D(  (\phi_n^+)^{q^* /2}   ) \|_{L^2((0,T)\times \T^d )}
 \le C \left( 1+  \| m_n\|^{q-1}_{L^q ((0,T)\times \T^d) } + \|u_T^+\|_{L^{q^*} (\T^d)} \right),
\end{equation}
and that 
\begin{displaymath}
    \|D(  (\phi_n^+)^{q^* /2}   ) \|^2_{L^2((0,T)\times \T^d )}\le  C \left( 1+  \| m_n\|^{q}_{L^q ((0,T)\times \T^d) } + \|u_T^+\|^{q^*}_{L^{q^*} (\T^d)} \right).
\end{displaymath}
\begin{remark}
  \label{sec:priori-estim-maxim-1}
Note that the latter estimate does not hold with a degenerate diffusion as in Remark \ref{sec:assumptions}, but it will not be used hereafter.
\end{remark}

This implies that
\begin{equation}\label{eq:41}
  \int_{\T^d } m_0(x)\phi_n^+ (0,x) dx \le  C \left(  \| m_n\|^{q-1}_{L^q ((0,T)\times \T^d) } +1\right ).
\end{equation}

% Note that
% $\| m_n^{q-1}\|_{L^2 ((\tau,T)\times \T^d) } \le C \| m_n\|^{q-1 }_{L^q ((\tau,T)\times \T^d) }$ because $1<q\le 2$.
% Multiplying (\ref{eq:39}) by $\phi_n^+ e^{-\lambda (T-t)} $ for $\lambda $ large enough, integrating in $(\tau,T)\times \T^d$ and using the latter observation, we obtain that
% \begin{equation}
%   \label{eq:31bis}
% \|\phi_n^+\|_{L^\infty(0,T;L^2(\T^d) )}+  \|D\phi_n^+\|_{L^2((0,T)\times \T^d )} \le C \left(  \| m_n\|^{q-1}_{L^q ((0,T)\times \T^d) } + \|u_T^+\|_{L^2 (\T^d)} \right).
% \end{equation}
% This implies that
% \begin{equation}\label{eq:32bis}
%   \int_{\T^d } m_0(x)\phi_n^+ (0,x) dx \le  C \left(  \| m_n\|^{q-1}_{L^q ((0,T)\times \T^d) } +1\right ).
% \end{equation}
Combining (\ref{eq:38}) and (\ref{eq:41}), we obtain that
\begin{equation*}
  \begin{split}
&\ds   \int_0^T \int_{\T^d}  \left(  \frac 1 {C_2} m_n^q(t,x) -C_2 m_n(t,x)  + m_n^{1-\alpha}(t,x) |D \phi_n(t,x)|^\beta  \right)  
dx dt  +\int_{\T^d}    m_0(x)\phi_n^-(0,x) dx\\ & \le -M +\epsilon + C     \left(  \| m_n\|^{q-1}_{L^q ((0,T)\times \T^d) } +1\right ).
  \end{split}
\end{equation*}
The latter and (\ref{eq:41}) yield 
\begin{equation}\label{eq:42}
\ds   \int_0^T \int_{\T^d}  \left(  m_n^q(t,x)   + m_n^{1-\alpha}(t,x) |D \phi_n(t,x)|^\beta  \right)  
dx dt  +\int_{\T^d}   |\phi_n(0,x)| dx\le C.
\end{equation}

Let $\bar m(t,x)= 1_{\tau <t}$: since $\cA (\phi_n, \bar m) \ge \cA(\phi_n)$, we obtain that
\begin{displaymath}
  \begin{array}[c]{rcl}
    M-\epsilon&\le &   \ds \int_0^t \int_{\T^d}   \left(\frac{\partial \phi_n} {\partial t} (\tau,x)  + \nu \Delta \phi_n(\tau,x)  +  H
(x, 1, D \phi_n(\tau,x) )          \right)    dx d  t    \\ & & \ds +\int_{\T^d}    m_0(x)\phi_n(0,x) dx\\
&=& 
 \ds \int_0^t \int_{\T^d}      H
(x, 1, D \phi_n(\tau,x) )        dx d\tau +\int_{\T^d}   \phi_n(t,x) dx +\int_{\T^d}   (m_0(x)-1) \phi_n(0,x) dx .
  \end{array}
\end{displaymath}
This implies that $t\mapsto \int_{\T^d} \phi_n(t,x) dx $ is bounded from below uniformly w.r.t. $n$. Combining with (\ref{eq:40}), we get that $ \phi_n^-$ is bounded 
in $L^\infty(0,T; L^1(\T^d))$, and finally that $ \phi_n$ is bounded 
in $L^\infty(0,T; L^1(\T^d))$.
\\
Finally, setting $\gamma_n^1 = -\ell(\cdot, m_n)-m_n \frac {\partial \ell}{\partial m} (\cdot, m_n)$ and $0\le \gamma_n^2= 
(1-\alpha ) \frac {|D\phi_n|^\beta}{m_n^\alpha} 1_{\{m_n>0\}}$,
we see that the sequence $(\gamma_n^1)_n$ is bounded in $L^{q^*}((0,T)\times \T ^d)$, and that 
\begin{displaymath}
\frac{\partial \phi_n} {\partial t}  + \nu \Delta \phi_n  \ge \gamma_n^1+\gamma_n^2.
\end{displaymath}
Integrating the latter on $(0,T)\times \T^d$, we obtain that 
\begin{displaymath}
   \int_0^T \int_{\T^d} \gamma_n^2(t,x)dxdt \le \int_{\T^d}  u_T(x)dx -  \int_{\T^d}  \phi_n(0,x)dx - \int_0^T \int_{\T^d} \gamma_n^1(t,x)dxdt,
\end{displaymath}
which implies that the sequence of positive function $ (\gamma_n^2)_n$ is bounded in $L^1((0,T)\times \T^d)$. 
\\
To summarize, we have proven the following lemma:
\begin{lemma}
  \label{sec:priori-estim-maxim-5}
The maximizing sequence $(\phi _n)_{n\in \N}$ introduced at the beginning of \S~\ref{sec:priori-estim-maxim-5} is uniformly bounded in $L^\beta(0,T; W^{1,\beta}(\T^d))$ and in $L^\infty(0,T; L^1(\T^d))$.
\\
Noting $m_n$  the nonnegative function achieving $\cA(\phi_n,m_n)= \cA(\phi_n)$, the sequence $(m_n)_{n\in \N }$ is uniformly bounded in $L^{q}((0,T)\times \T^d)$.\\
Calling $\gamma_n(t,x)=  -H(x, m_n(t,x), D\phi_n(t,x)) - m_n(t,x)H_m(x, m_n(t,x), D\phi_n(t,x))$,
with the convention that 
$ H(x, m, p )     + m H_m(x, m,p )=\ell(x,0)$ if  $m=0$  and $p=0$,
%\frac {\partial \phi_n} {\partial t}  + \nu \Delta \phi_n$, 
we can split $\gamma_n$ as follows: $\gamma_n= \gamma_n^1+ \gamma_n^2$,
where $\gamma_n^1 = -\ell(\cdot, m_n)-m_n \frac {\partial \ell}{\partial m} (\cdot, m_n)$ and $ \gamma_n^2= 
(1-\alpha ) \frac {|D\phi_n|^\beta}{m_n^\alpha} 1_{\{m_n>0\}}$. The sequence $(\gamma_n^1)_{n\in \N}$ is uniformly  bounded in $L^{q*}((0,T)\times \T^d)$. The function $\gamma_n^2$ is nonnegative and the sequence $(\gamma_n^2)_{n\in \N}$ is uniformly bounded in $L^{1}((0,T)\times \T^d)$.
\end{lemma}

\section{A relaxed problem}
\label{sec:relaxed-problem}

\subsection{Definition and first properties}
\label{sec:defin-first-prop}
Let $\cK$ be the set of pairs $(\phi, \gamma)$ 
such that 
\begin{itemize}
\item $\phi \in L^\beta(0,T; W^{1,\beta}(\T^d))\cap  L^\infty(0,T; \cM(\T^d))$
%and $\int_{0}^T\int_{\T^d}   \frac {m^* (t,x)} { 1+ (m^* (t,x))^\alpha  } |D\phi(t,x)|^\beta dxdt<+\infty$, where $(m^*, z^*)$ is the solution of (\ref{eq:23}) found in Lemma~\ref{sec:two-optim-probl-2}. 
\item $\gamma\in \cM([0,T]\times \T^d)$ and $\gamma ^{\rm{ac}}= \gamma^1+ \gamma^2$, with $\gamma^1\in L^{q^*}((0,T)\times \T^d)$ is non positive, $\gamma^2$ is non negative and  $\gamma^2\in L^{1}((0,T)\times \T^d)$
,  $\gamma ^{\rm{sing}}\in \cM_+([0,T]\times \T^d)$
\item \begin{equation}
 \label{eq:43}
   \frac{\partial \phi} {\partial t}   + \nu \Delta \phi\ge \gamma, \quad \hbox{and }\quad \phi|_{t=T}\le u_T.
\end{equation}
\end{itemize}
It is clear that $\cK$ is convex.
The following lemma implies that $\phi$ has a trace in a very weak sense:
\begin{lemma}
  \label{sec:relaxed-problem-1}
Consider $(\phi,\gamma)\in \cK$. For any Lipschitz continuous map $\xi: \T^d\to \R$, the map $t\mapsto \int_{\T^d} \xi(x) \phi(t, x) dx$ has a BV representative 
on $[0,T]$. Moreover, if we note $\int_{\T^d} \xi(x)\phi(t^+, x) dx$ its right limit at $t\in [0,T)$, then the map $\xi \mapsto \int_{\T^d} \xi(x)\phi(t^+, x) dx$ 
can be extended to a  bounded linear form on  $\cC^0(\T^d)$.
\end{lemma}
\begin{proof}
Consider first a nonnegative and Lipschitz continuous function $\xi : \T^d\to \R_+$; the following identity holds in the sense of distributions:
\begin{displaymath}
  -\frac {d}{dt}\left(t\mapsto\int _{\T^d} \xi(x) \phi(t, x) dx\right)- \nu\int_{\T^d}\nabla \phi(t,x) \cdot \nabla \xi(x) dx \ge  \int _{\T^d} \xi(x) \gamma ^{\rm{ac}}(t, x) dx, 
\end{displaymath}
The second and third integral in the latter inequality belong to $L^1 ((0,T))$. From this, we deduce that $t\mapsto \int_{\T^d} \xi(x) \phi(t, x) dx$ has a BV representative.
% We obtain that 
% \begin{displaymath}
%   \int _{\T^d} \xi(x) \phi(t_2^-, x) dx  -\int _{\T^d} \xi(x) \phi(t_1^+, x) dx \le -\nu \int_{t_1}^{t_2}\int_{\T^d}\nabla \phi(s,x) \cdot \nabla \xi(x) dx ds
% -\int_{t_1}^{t_2} \int _{\T^d} \xi(x) \gamma ^{1}(s, x) dx ds.
% \end{displaymath}
% Then 
% \begin{displaymath}
%   \int _{\T^d} \xi(x) \phi(t_2^-, x) dx\le \int _{\T^d} \xi(x) \phi(t_1^+, x) dx + C\left( (t_2-t_1)^{\frac {p-1}p } +(t_2-t_1)^{\frac 1 q }\right).
% \end{displaymath}
\\
If now $\xi$ is a Lipschitz continuous function that may change sign, $\xi : \T^d\to \R$, then we write $\xi=\xi^+-\xi^-$ and use the above argument separately for $\xi^+$ and $\xi^-$: we still obtain that  $t\mapsto \int_{\T^d} \xi(x) \phi(t, x) dx = \int_{\T^d} \xi^+(x) \phi(t, x) dx  - \int_{\T^d} \xi^-(x) \phi(t, x) dx $  has a BV representative.\\ 
The continuity of $\xi \mapsto \int_{\T^d} \xi(x)\phi(t^+, x) dx$ comes from the fact that  $\phi \in  L^\infty(0,T; \cM(\T^d))$.
\end{proof}
Thanks to Lemma~\ref{sec:relaxed-problem-1}, we may  define the concave functional $J$ on $\cK$ by 
\begin{equation}
  \label{eq:44}
J(\phi,\gamma)= \int_0^T \int_{\T^d} K(x, \gamma^{\rm ac}(t,x), D\phi(t,x)) dxdt +\int_{\T^d} m_0(x) \phi(0^+,x) dx, 
\end{equation}
and the relaxed optimization problem: 
\begin{equation}
  \label{eq:45}
\sup_{(\phi,\gamma)\in \cK} J(\phi,\gamma).
\end{equation}
In (\ref{eq:44}), note that $ (x,t)\mapsto K(x, \gamma^{\rm ac}(t,x), D\phi(t,x))$    is a measurable nonpositive function, so the first integral is meaningful and has a value in $[-\infty, 0]$.\\
Note also that, from (\ref{eq:30}), it is possible to restrict ourselves to the pairs $(\phi,\gamma)\in \cK$ such that 
$     \gamma^{\rm ac} (t,x)\le  -\ell (x,0)$   for almost every $(t,x)$ such that  $ D\phi(t,x)=0$.  Noting $\widetilde \cK$ the set
\begin{displaymath}
  \widetilde \cK= \left\{   (\phi,\gamma)\in \cK: 
     \gamma^{\rm ac} (t,x)\le  -\ell (x,0) \hbox{ for almost every }(t, x) \hbox{ s.t. }    D\phi(t,x)=0    \right\},
\end{displaymath}
we have 
\begin{equation}
\label{eq:46}
\sup_{(\phi,\gamma)\in \cK} J(\phi,\gamma)= \sup_{ (\phi,\gamma)\in \widetilde \cK}
 J(\phi,\gamma).    
\end{equation}
% Moreover, if $\gamma^{\rm ac} =\gamma^1+\gamma^2 \le -H(x, 0, D\phi)  =  |D\phi|^\beta  - \ell(x,0)$, with $0\le \gamma^2 \in L^1((0,T)\times \T^d)$, 
% $0\ge \gamma^1 \in L^{q^*}((0,T)\times \T^d)$, it is always possible to write $ \gamma^{\rm ac} =\tilde \gamma^1+\tilde \gamma^2$, 
% with $\tilde \gamma^2= \gamma^2 + (\gamma^1 +\ell(x,0))^+$ and  $\tilde \gamma^1= \gamma^1 - (\gamma^1 +\ell(x,0))^+$.
% We get that  $ \tilde \gamma^2 \in L^1((0,T)\times \T^d)$, $ 0\le \tilde \gamma^2\le  |D\phi|^\beta $,  $\tilde \gamma^1 \in L^{q^*}((0,T)\times \T^d)$
% and $\tilde \gamma^1\le  - \ell(x,0) \le 0$.\\
% Therefore 
% \begin{equation}
% \label{eq:47}
% \sup_{(\phi,\gamma)\in \cK} J(\phi,\gamma)= \sup_{.
%   \begin{array}[c]{l}
% (\phi,\gamma)\in \cK,\;     \gamma^{\rm ac} =  \gamma^1+ \gamma^2,
% \\     0\le \gamma^2\le  |D\phi|^\beta 
% \\  \gamma^1\le  - \ell(x,0)
%   \end{array}
%    } J(\phi,\gamma).
% \end{equation}

\begin{lemma}
  \label{sec:relaxed-problem-3}
For any $(\phi,\gamma)\in \widetilde \cK$  such that
\begin{displaymath}
  \int_{0}^T\int_{\T^d} K(x, \gamma^{\rm ac}(t,x), D\phi (t,x)) dxdt >- \infty,
\end{displaymath}
%  $  \gamma^{\rm ac} = \gamma^1 + \gamma^2$   with   $0\le \gamma^2\le  |D\phi|^\beta $ and $\gamma^1\le  - \ell(x,0)$,
 for any $(m,z)\in \cK_1$ such that $m\in L^q((0,T)\times \T^d)$ and 
\[\ds \int_0^T \int_{\T^d} \tilde L (x, m( t,x), z(t,x)) dxdt <+\infty,\]
%  and
% \begin{equation}
%   \label{eq:48}
% \int_{0}^T\int_{\T^d}   m^{1-\alpha}(t,x) |D\phi(t,x)|^\beta dxdt<+\infty,
% \end{equation}
the following holds:\\
 for almost any $t\in (0,T)$, 
\begin{equation}
  \label{eq:49}
  \begin{split}
&-\int_{\T^d} m(T,x)   u_T( x)  dx +\int_{\T^d} m(t,x)   \phi(t, x)  dx
+\int_{s=t}^T \int_{\T^d} m(s,x)  \left (\gamma ^{\rm ac} (s,x) + H(x,m(s,x), D\phi(s,x) ) \right)   dx ds\\  \le& 
\int_{s=t}^T \int_{\T^d} \tilde L (x, m(s,x),z(s,x))dx ds,    
  \end{split}
\end{equation}
\begin{equation}
\label{eq:50}
\begin{split}
& -\int_{\T^d} m(t,x)   \phi(t, x)  dx +\int_{\T^d} m_0(x)   \phi(0^+, x)  dx
+\int_{s=0}^t \int_{\T^d} m(s,x)  \left (\gamma ^{\rm ac} (s,x) + H(x,m(s,x), D\phi(s,x) ) \right)     dx ds \\ \le & \ds 
\int_{s=0}^t \int_{\T^d} \tilde L (x, m(s,x),z(s,x))dx ds,  
\end{split}
\end{equation}
and the meaning of $\int_{\T^d} m(t,x)   \phi(t, x)  dx$ will be explained in the proof.\\
Moreover, if 
\begin{equation}
  \label{eq:51}
\begin{split}
&  -\int_{\T^d} m(T,x)   u_T( x)  dx +\int_{\T^d} m_0(x)   \phi(0^+, x)  dx
+\int_{s=0}^T \int_{\T^d} m(s,x) 
 \left (\gamma ^{\rm ac} (s,x) + H(x,m(s,x), D\phi(s,x) ) \right)      dx ds \\  =& \ds
\int_{s=0}^T\int_{\T^d} \tilde L (x, m(s,x),z(s,x))dx ds,
  \end{split}
\end{equation}
  then $z(t,x)=m(t,x)H_p(x, m(t,x),D\phi(t,x))$ holds almost everywhere.
\end{lemma}

\begin{remark}\label{sec:defin-first-prop-1}
  Before proving Lemma \ref{sec:relaxed-problem-3}, note that the integrals of \\
$(s,x)\mapsto m(s,x)  \left (\gamma ^{\rm ac} (s,x) + H(x,m(s,x), D\phi(s,x) ) \right)$ in (\ref{eq:49})-(\ref{eq:51})
have a meaning in $(-\infty, +\infty]$, because  
$ m(s,x)  \left (\gamma ^{\rm ac} (s,x) + H(x,m(s,x), D\phi(s,x) ) \right) \ge K(x, \gamma ^{\rm ac} (s,x),  D\phi(s,x))$ a.e. and from the hypothesis of Lemma~\ref{sec:relaxed-problem-3}, the latter  is integrable in  $(0,T)\times \T^d$ with a nonpositive integral.
\end{remark}

\begin{proof}
  We first extend $m$ to $[-1,T+1]$ by setting $m(t)= m_0$ for $t\le 0$ and $m(t) =m(T)$ for $t>T$.
 Note that $m(T)$ is well defined from Lemma \ref{sec:two-optim-probl-1}. Similarly, we extend   
 $z$ to $[-1,T+1]$ by setting $z(t)= 0$ if $t\notin (0,T)$.\\

Consider a regularizing kernel $\eta_\epsilon(t,x)= \epsilon^{-d-\delta } h_1(\frac t {\epsilon^\delta}) h_2(\frac x {\epsilon}) $ where $h_1$ is a smooth even and nonnegative function supported in $[-1/2,1/2]$ such that $\int_\R h_1(t)dt=1$, $h_2$ is a smooth symmetric nonnegative function  supported in $[-1/2,1/2]^d$  such that $\int_{\R^d} h_2(x)dx=1$, and $\delta$ will be chosen later.
% Let $\eta_\epsilon(t,x)= \epsilon^{-1-d} \eta(\frac t \epsilon, \frac x \epsilon) $ be a positive smoothing kernel with mass $1$,
% supported in the ball $B(0,\epsilon)$ of $\R^{d+1}$.
We define $m_\epsilon= \eta_\epsilon \star m$ and $z_\epsilon= \eta_\epsilon \star z$ in $(-1/2, T+1/2)\times \T^d$.
We can see that in $(0,T)\times \T^d$, 
\begin{equation}
  \label{eq:52}
 \frac{\partial m_\epsilon} {\partial t}   - \nu \Delta  \left (m_\epsilon \;1_{t\in (0,T)}\right) +  \diver   z_\epsilon  =\diver R_\epsilon   
\end{equation}
where $R_\epsilon= \eta_\epsilon \star \left (D m \;1_{t\in (0,T)}\right)- 1_{t\in (0,T)} D (\eta_\epsilon\star m)$. 
From \cite{MR1022305}, we know that $R_\epsilon\to 0$ in $L^q( (0,T)\times \T^d)$ as $\epsilon\to 0$, 
because $m\in L^q( (0,T)\times \T^d)$.
\\
We know that $\phi(T)\le u_T$. Thus $\int_{\T^d} \phi(T, x) m_\epsilon(T,x) dx \le \int_{\T^d} u_T( x) m_\epsilon(T,x) dx$ and the latter integral converges to
 $\int_{\T^d} u_T( x) m(T,x) dx$ in view of Lemma \ref{sec:two-optim-probl-1}.
\\
We know that $m_\epsilon\to m$ in $L^q ( (0,T)\times \T^d)$ and that $\phi\in L^\beta   ( (0,T)\times \T^d)$. 
From Assumption (H3), %Since $\beta> q^*$, 
this implies that
 $\phi m_\epsilon\to \phi m$ in $L^1 ( (0,T)\times \T^d)$. 
\\ 
From  Lemma \ref{sec:relaxed-problem-1}, we know that  out of a countable set, $\phi(t^-)=\phi(t^+)=\phi(t)$ and that  
$\int_{\T^d} \phi(t, x) m_\epsilon(t,x) dx $ is well defined. \\
 From the latter two observations,    up to the extraction of a subsequence, 
we may assume that \[ \int_{\T^d} \phi(t, x) m_\epsilon(t,x) dx  \to \int_{\T^d} \phi(t, x) m(t,x) dx \]
as $\epsilon\to 0$,  with  $\phi(t^-)=\phi(t^+)=\phi(t)$,
 for almost all $t\in (0,T)$.  Let $t\in (0,T)$ be such  that the latter is true. \\
Since $(\phi, \gamma)\in \cK$,
\begin{equation}
  \label{eq:53}
  \begin{split}
&\ds \int_t^T \int_{\T^d} \left(\phi(s,x) \frac {\partial m_\epsilon}{\partial t} (s,x) +\nu D\phi\cdot D(1_{s\in (0,T)} m_\epsilon (s,x) ) \right) dxds\\
&\ds- \int_{\T^d} u_T( x) m_\epsilon(T,x) dx + \int_{\T^d} \phi(t, x) m_\epsilon(t,x) dx 
\\
\le & - \int_t^T \int_{\T^d}  \gamma^{\rm ac } (s,x)m_\epsilon(s,x) dx ds    
  \end{split}
\end{equation}
On the other hand, from (\ref{eq:52}), 
\begin{equation}
  \label{eq:54}
  \begin{split}
&\ds \int_t^T \int_{\T^d} \left(\phi(s,x) \frac {\partial m_\epsilon}{\partial t} (s,x) +\nu D\phi\cdot D(1_{s\in (0,T)} m_\epsilon (s,x) ) \right) dxds\\
= & - \int_t^T \int_{\T^d}   D\phi(s,x) \cdot R_\epsilon(s,x)  dx ds     +\int_t^T \int_{\T^d}   D\phi(s,x) \cdot z_\epsilon(s,x)  dx ds    
  \end{split}
\end{equation}
We deduce from (\ref{eq:53}) and (\ref{eq:54}) that
\begin{equation}
  \label{eq:55}
  \begin{split}
&- \int_t^T \int_{\T^d}   D\phi(s,x) \cdot R_\epsilon(s,x)  dx ds     +\int_t^T \int_{\T^d}   D\phi(s,x) \cdot z_\epsilon(s,x)  dx ds   \\
&\ds -\int_t^T \int_{\T^d} m_\epsilon(s,x) H(x,m_\epsilon(s,x) , D\phi (s,x) ) ds ds
- \int_{\T^d} u_T( x) m_\epsilon(T,x) dx +\int_{\T^d} \phi(t, x) m_\epsilon(t,x) dx 
 \\
\le &  - \int_t^T \int_{\T^d}  \gamma^{\rm ac } (s,x)m_\epsilon(s,x) dx ds    -\int_t^T \int_{\T^d} m_\epsilon(s,x) H(x,m_\epsilon(s,x) , D\phi (s,x) ) ds ds.
  \end{split}
\end{equation}
Finally, we use the fact that 
\begin{equation}
  \label{eq:56}
  \begin{split}
&\ds   \int_t^T \int_{\T^d}  \left( D\phi(s,x) \cdot z_\epsilon(s,x)  -m_\epsilon(s,x) H(x,m_\epsilon(s,x) , D\phi (s,x) ) \right)dx ds \\ \ge &\ds   -\int_t^T \int_{\T^d}\tilde L
(x,m_\epsilon(s,x) , z_\epsilon (s,x) ) dx ds.     
  \end{split}
\end{equation}
Combining this with (\ref{eq:55}), we find that
\begin{equation}
  \label{eq:57}
  \begin{split}
&\ds - \int_{\T^d} u_T( x) m_\epsilon(T,x) dx +\int_{\T^d} \phi(t, x) m_\epsilon(t,x) dx   
\\
\le  &- \int_t^T \int_{\T^d} m_\epsilon(s,x) \left(   \gamma^{\rm ac } (s,x) + H(x,m_\epsilon(s,x) , D\phi (s,x) ) \right) ds ds \\ & \ds +
\int_t^T \int_{\T^d}\tilde L
(x,m_\epsilon(s,x) , z_\epsilon (s,x) ) dx ds   +\int_t^T \int_{\T^d}   D\phi(s,x) \cdot R_\epsilon(s,x)  dx ds.     
  \end{split}
\end{equation}
\begin{enumerate}
\item We have seen that the first line of (\ref{eq:57}) tends to $- \int_{\T^d} u_T( x) m(T,x) dx +\int_{\T^d} \phi(t, x) m(t,x) dx  $.
\item
From Remark   \ref{sec:defin-first-prop-1},   or more precisely 
the facts that $ \int_{0}^T\int_{\T^d} K(x, \gamma^{\rm ac}(t,x), D\phi (t,x)) dxdt >- \infty$ and that
\begin{displaymath}
  \gamma^{\rm ac } (s,x)   +H(x,m_\epsilon(s,x) , D\phi (s,x) ) m_\epsilon(s,x) \ge K(x, \gamma^{\rm ac } (s,x) , D\phi (s,x)  ),
\end{displaymath}
we can use  Fatou lemma and  get that 
\begin{displaymath}
  \begin{split}
    &\int_t^T \int_{\T^d}  \left(\gamma^{\rm ac } (s,x)        +H(x,m(s,x) , D\phi (s,x) ) \right)        m(s,x) dx ds 
\\ \le &  \liminf_{\epsilon\to 0}  \int_t^T \int_{\T^d}   \left( \gamma^{\rm ac } (s,x)    
+H(x,m_\epsilon(s,x) , D\phi (s,x) )\right)
m_\epsilon(s,x) dx ds .
  \end{split}
\end{displaymath}
\item $\int_t^T \int_{\T^d}   D\phi(s,x) \cdot R_\epsilon(s,x)  dx ds$ tends to $0$ because $R_\epsilon$ tends to $0$ in $L^q((0,T)\times \T^d)$ and $D\phi\in L^\beta((0,T)\times \T^d))$ with $\beta\ge q^*$.
\item Finally, from the convexity of $\tilde L$ with respect to $(m, z)$ and from (\ref{eq:11}), we see that 
\begin{displaymath}
  \limsup_{\epsilon\to 0} \int_t^T \int_{\T^d}\tilde L(
x,m_\epsilon(s,x) , z_\epsilon (s,x) ) dx ds \le 
 \int_t^T \int_{\T^d}\tilde L
(x,m(s,x) , z(s,x) ) dx ds .
\end{displaymath}
\end{enumerate}
From (\ref{eq:57}) and all the points above, we deduce (\ref{eq:49}).
\\
 Similarly as for (\ref{eq:57}), we obtain that
\begin{equation}
\label{eq:58}
  \begin{split}
&\ds - \int_{\T^d} \phi(t, x) m_\epsilon(t,x) dx  +\int_{\T^d} \phi(0, x) m_\epsilon(0,x) dx 
\\
\le  &- \int_0^t \int_{\T^d}  \gamma^{\rm ac } (s,x)m_\epsilon(s,x) dx ds    -\int_0^t \int_{\T^d} m_\epsilon(s,x) H(x,m_\epsilon(s,x) , D\phi (s,x) ) dx ds \\ & \ds +
\int_0^t \int_{\T^d}\tilde L
(x,m_\epsilon(s,x) , z_\epsilon (s,x) ) dx ds +\int_0^t \int_{\T^d}   D\phi(s,x) \cdot R_\epsilon(s,x)  dx ds,
  \end{split}
\end{equation}
in which $\int_{\T^d} \phi(0, x) m_\epsilon(0,x) dx $ has a meaning from Lemma \ref{sec:relaxed-problem-1}. 
We claim that $\int_{\T^d} \phi(0, x) m_\epsilon(0,x) dx\to \int_{\T^d} \phi(0, x) m_0(x) dx$ as $\epsilon\to 0$: 
indeed,  let $\zeta\in (0,1)$ be the H{\"o}lder exponent in Lemma \ref{sec:two-optim-probl-1}: in view of Remark \ref{sec:two-optim-probl-3}, we know that 
$\zeta \ge \min( 1/2, (1-\alpha)/\beta)$;
calling $h_{2,\epsilon} (x)= \epsilon^{-d} h_2(\frac x \epsilon)$ and  $h_{1,\epsilon} (t)= \epsilon^{-\delta} h_1(\frac t {\epsilon^\delta})$, we get that for all $x\in \T ^d$,
\begin{displaymath}
  \begin{split}
  |m_\epsilon(0,x)-  (h_{2,\epsilon} \star m_0)(x)| &=   \left|\int \int    ( m(s,y)-  m_0(y))  h_{1,\epsilon}(-s) h_{2,\epsilon}(x-y) ds dy \right|\\
& \le  C \int s^{\delta \zeta -d -1} h_{1,\epsilon}(-s) ds   \le C \epsilon ^{\delta \zeta -d -1 }.
  \end{split}
\end{displaymath}
Choosing $\delta$ large enough, (i.e. such that $\delta \zeta -d -1>0$) and 
%where the last line is a consequence of the proof of Lemma 3.1 in \cite{cardaliaguet2014second}, see also Lemma \ref{sec:two-optim-probl-1}.
using the fact that  $m_0$ is $\cC^1$,
\begin{displaymath}
  \lim_{\epsilon \to 0} \| h_{2,\epsilon} \star m_0  -m_0\|_{L^\infty (\T^d)} = 0. 
\end{displaymath}
Therefore, $ \lim_{\epsilon \to 0} \| m_\epsilon(0,\cdot) -m_0\|_{L^\infty (\T^d)} = 0$.
The claim follows from  the continuity stated in Lemma  \ref{sec:relaxed-problem-1}.
% we can write 
% \begin{displaymath}
%   \begin{split}
%     \left| \int_{\T^d} \phi(0, x) m_\epsilon(0,x) dx- \int_{\T^d} \phi(0, x) m_0(x) dx\right| &=  \left|  \int_{\T^d} (h_{2,\epsilon} \star \phi)(0, x)   (h_{1,\epsilon} \star 
% m) (0,x) dx -\int_{\T^d} \phi(0, x) m_0(x) dx\right| \\
% &\le  \left|  \int_{\T^d} (h_{2,\epsilon} \star \phi)(0, x)   (h_{1,\epsilon} \star 
% m -m_0) (0,x) dx  \right|  \\ &\phantom{\le} + \left|   \int_{\T^d}    \left((h_{2,\epsilon} \star \phi)(0, x)  -\phi(0, x) \right) m_0(x) dx\right| 
%   \end{split}
% \end{displaymath}
% where $h_{2,\epsilon} (x)= \epsilon^{-d} h_2(\frac x \epsilon)$ and  $h_{1,\epsilon} (t)= \epsilon^{-\delta} h_2(\frac t {\epsilon^\delta})$.
% From the proof of Lemma 3.1 in [], we see that 
% \begin{displaymath}
%   \left|  \int_{\T^d} (h_{2,\epsilon} \star \phi)(0, x)   (h_{1,\epsilon} \star 
% m -m_0) (0,x) dx  \right| \le C h^{\delta /2} \| h_{2,\epsilon} \star \phi (0,\cdot)\|_{Lip(\T^d)} \le C h^{\frac \delta 2  -d-1} \to 0.
% \end{displaymath}
% On the other hand 
% \begin{displaymath}
%       \int_{\T^d}    (h_{2,\epsilon} \star \phi)(0, x)   m_0(x) dx=   \int_{\T^d}     \phi(0, x)   (h_{2,\epsilon} \star m_0)(x) dx  \to   
% \int_{\T^d}     \phi(0, x)    m_0(x) dx
% \end{displaymath}
% from  the continuity stated in Lemma  \ref{sec:relaxed-problem-1}. We have proved the claim.
\\
The fact that $\int_{\T^d} \phi(0, x) m_\epsilon(0,x) dx\to \int_{\T^d} \phi(0, x) m_0(x) dx$  and the arguments above imply (\ref{eq:50}).
\\
Let us now suppose that (\ref{eq:51}) holds: then the inequalities in (\ref{eq:49}) and (\ref{eq:50}) are equalities, for almost all $t$. 
For $\sigma>0$ let us introduce the set 
\begin{displaymath}
  E_{\sigma}(t)= \left\{
    \begin{array}[c]{l}
(s,y)\in [t,T]\times \T^d \; : \;  \\
m(s,y) H(x,m(s,y) , D\phi (s,y) ) -\tilde L
(y,m(s,y) , z (s,y) )\le  z (s,y) \cdot D\phi (s,y)  - \sigma
          \end{array}
\right\}.
\end{displaymath}
If $| E_{\sigma}(t)|>0$, then for $\epsilon>0$ small enough, $| E_{\epsilon,\sigma}(t)|>| E_{\sigma}(t)| /2$, where
\begin{displaymath}
  E_{\epsilon,\sigma}(t)= \left\{
    \begin{array}[c]{l}
(s,y)\in [t,T]\times \T^d \; : \;  \\
m_\epsilon(s,y) H(x,m_\epsilon(s,y) , D\phi (s,y) ) -\tilde L
(y,m_\epsilon(s,y) , z_\epsilon (s,y) )\le  z_\epsilon (s,y) \cdot D\phi (s,y)  - \frac \sigma 2
          \end{array}
\right\}.
\end{displaymath}
Then (\ref{eq:56}) becomes 
\begin{equation}
  \begin{split}
&\ds   \int_t^T \int_{\T^d}  \left( D\phi(s,x) \cdot z_\epsilon(s,x) - m_\epsilon(s,x) H(x,m_\epsilon(s,x) , D\phi (s,x) ) \right)dx ds \\ \ge &\ds   -\int_t^T \int_{\T^d}\tilde L
(x,m_\epsilon(s,x) , z_\epsilon (s,x) ) dx ds     +\frac \sigma  4    |E_{\sigma}(t)|,     
  \end{split}
\end{equation}
which implies that 
\begin{displaymath}
  \begin{split}
&-\int_{\T^d} m(T,x)   u_T( x)  dx +\int_{\T^d} m(t,x)   \phi(t, x)  dx
+\int_{s=t}^T \int_{\T^d} m(s,x) \left(\gamma ^{\rm ac} (s,x)+ H(x,m(s,x), D\phi(s,x) ) \right)  dx ds \\  \le& 
\int_{s=t}^T \int_{\T^d} \tilde L (x, m(s,x),z(s,x))dx ds     -\frac \sigma  4    |E_{\sigma}(t)|,  
  \end{split}
\end{displaymath}
  in contradiction with the fact that there is an equality in (\ref{eq:49}).
\\
Hence, $m(s,y) H(x,m(s,y) , D\phi (s,y) ) -\tilde L
(y,m(s,y) , z (s,y) )=  z (s,y) \cdot D\phi (s,y)$  holds almost everywhere. In view of (\ref{eq:22}), this shows that $ z (s,y)=  m(s,y) H_p(x,m(s,y) , D\phi (s,y) )$ a.e. in $\{m>0\}$. Furthermore, from (\ref{eq:21}) and the fact that $\ds \int_0^T \int_{\T^d} \tilde L (x, m( t,x), z(t,x)) dxdt <+\infty$, we see that
 $z(s,y)=0$  a.e. in  $\{m=0\}$. Hence, $ z (s,y)=  m(s,y) H_p(x,m(s,y) , D\phi (s,y) )$ a.e. in $(0,T)\times \T^d$.
\end{proof}
\begin{proposition}
\label{sec:relaxed-problem-2}
  \begin{displaymath}
\sup_{(\phi,\gamma)\in \cK} J(\phi,\gamma)=\sup_{\phi\in \cK_0} \cA(\phi).    
  \end{displaymath}
\end{proposition}
\begin{proof}
   It is clear that $\forall \phi\in \cK_0$, $-\infty< \cA(\phi)\le \sup_{(\phi,\gamma)\in \cK} J(\phi,\gamma)$: indeed consider $m$ achieving $\cA(\phi, m)=\cA(\phi)$,
i.e. $m=\psi\left (\frac{\partial \phi} {\partial t}  + \nu \Delta \phi, D\phi  \right)$. We know that $ \int_0^T \int_{\T^d} m^{1-\alpha} |D\phi|^\beta <+\infty$ and that 
 $ \int_0^T \int_{\T^d} m^q <\infty$.
 Let us take $\gamma= \gamma^1+\gamma^2$ where $\gamma^1 = -\ell(\cdot, m)-m \frac {\partial \ell}{\partial m} (\cdot, m)$ and
 $\gamma^2= \frac { (1-\alpha) }{m^{\alpha}} |D\phi|^\beta 1_{m>0}$. It is easy to check that $(\phi,\gamma)\in \cK$ and that $\cA(\phi)=J(\phi, \gamma)$. Hence $\sup_{(\phi,\gamma)\in \cK} J(\phi,\gamma)\ge \sup_{\phi\in \cK_0} \cA(\phi)$.\\
For the reverse inequality, consider $(\phi,\gamma)\in \widetilde \cK$ 
 such that  %   $\gamma^{\rm ac} + H(x,0, D\phi)\le 0$ almost everywhere, and that 
$J(\phi,\gamma)>-\infty$. This implies that \[\ds \int_0^T \int_{\T^d} K(x, \gamma^{\rm ac}(s,x),   D\phi(s,x) ) dx ds>-\infty.\]
Let $(m^* , z^*)$ be the pair of functions achieving (\ref{eq:23}), see Lemma~\ref{sec:two-optim-probl-2}.
% and $(m,z)\in \cK_1$ such that $m\in L^q((0,T)\times \T^d)$ and 
% $\int_0^T \int_{\T^d} \tilde L (x, m( t,x), z(t,x)) dxdt <+\infty$.
 From Lemma \ref{sec:relaxed-problem-3}, 
\begin{displaymath}
  \begin{split}
    J(\phi,\gamma)& \le \ds \int_0^T \int_{\T^d}  m^*(t,x)\left ( \gamma^{\rm ac}(t,x) + H(x, m^*(t,x), D\phi(t,x))\right)dx dt +\int_{\T^d} m_0(x) \phi(0^+,x) dx\\
    &\le \int_{s=0}^T \int_{\T^d} \tilde L (x, m^* (s,x),z^* (s,x))dx ds + \int_{\T^d} m^*(T,x)   u_T( x)  dx \\
    &= \cB(m^*,z^*).
  \end{split}
\end{displaymath}
Hence,
%Taking the $min$ with respect to $(m,z)$ in the right hand side, we obtain thanks to Lemma~\ref{sec:two-optim-probl-2} that 
\begin{displaymath}
   J(\phi,\gamma)\le \min_{(m,z)\in \cK_1} \cB(m,z) =\sup_{\phi\in \cK_0} \cA(\phi).
\end{displaymath}
and we conclude using (\ref{eq:46}).
\end{proof}

\subsection{Existence of a solution of the relaxed problem}
\label{sec:exist-solut-relax}
\begin{proposition}
  \label{sec:exist-solut-relax-1}
The relaxed problem (\ref{eq:45}) has at least a solution $(\phi, \gamma)\in \widetilde \cK$.% such that $\gamma^{\rm ac}\le -H(x,0,D\phi)$ a.e..
\end{proposition}
\begin{proof} $\;$
\paragraph{Step 1}
Consider the maximizing sequence $\phi_n$ for problem (\ref{eq:19}) described  in Lemma \ref{sec:priori-estim-maxim-5} and call $m_n$ the function such that $\cA(\phi_n)= \cA(\phi_n, m_n)$. 
Also, let the functions $\gamma_n^1$ and $\gamma_n ^2$ be  defined  as in Lemma \ref{sec:priori-estim-maxim-5}. We know that
\begin{displaymath}
\frac {\partial \phi_n} {\partial t}  + \nu \Delta \phi_n = \gamma_n\ge \gamma_n^1+ \gamma_n^2,
\end{displaymath}
and $(\phi_n,\gamma_n)\in \widetilde K$. The definition of $\gamma_n^1$ ensures that $ \gamma_n^1\le -\ell(\cdot,0)$.
%We have that $ H(x, 0, D\phi_n (t,x)) \le -\gamma_n(t,x)$ for all $(t,x)$.  
\\
Up to the  extraction of  a subsequence, we may assume that $\phi_n\rightharpoonup \phi $ in $L^\beta(0,T, W^{1,\beta}(\T^d))$,  $m_n\rightharpoonup m$ in $L^q((0,T)\times \T^d)$,  $\gamma_n^1  \rightharpoonup \gamma^1$ in $L^{q^*}((0,T)\times \T^d)$,  with $ \gamma^1\le -\ell(\cdot,0)$ a.e.,
 $\gamma_n^2  \rightharpoonup \tilde \gamma^2$ in $\cM((0,T)\times \T^d)$ $*$. It is clear that $\tilde \gamma^2$ is a positive Radon measure that we write $\tilde \gamma^2= \gamma^2+ \gamma^{\rm sing}$. We see that 
\begin{displaymath}
\frac {\partial \phi} {\partial t}  + \nu \Delta \phi \ge \gamma= \gamma^1+ \gamma^2.  
\end{displaymath}
Thus for any function $\xi$ in $Lip(\T^d)$, the function $t\mapsto \int_{\T^d} \xi(x) \phi (t,x) dx$ has a BV representative.
Moreover, since $t\mapsto \zeta_n(t)=\int_{\T^d}\xi(x) \phi_n (t,x) dx$ is bounded in $L^\infty((0,T))$ by $C \|\xi\|_{L^\infty(\T^d)}$, 
where $C$ is independent of $\xi$,
we can assume  (up to the extraction of  a subsequence) 
 that $\zeta_n\rightharpoonup \zeta$ in $L^\infty((0,T))$  weak *, and  $\|\zeta\|_{L^\infty((0,T))} \le C \|\xi\|_{L^\infty(\T^d)}$. 
Since for all smooth function $\psi: [0,T]\to \R$, we know that $\int_0^T \int_{\T^d} \xi(x) \phi_n (t,x) \psi(t) dx dt$ tends to $\int_0^T \int_{\T^d} \xi(x) \phi (t,x) \psi(t)dx dt$,
we see that $ \zeta= \int_{\T^d}\xi(x) \phi (\cdot,x) dx$. Therefore $\left|\int_{\T^d} \xi(x) \phi (t,x) dx\right |\le C \|\xi\|_{L^\infty(\T^d)}$. This shows that
$(\phi,\gamma)\in \cK$.
\paragraph{Step 2}
It can be proved that the map $ (\phi,\gamma)\in \cK  \mapsto - \int_0^T \int_{\T^d} K(x,  -\gamma^{\rm ac}(t,x), -D\phi(t,x)) dxdt$ is the 
restriction to $ \cK$ of the  convex conjugate of the map 
\begin{displaymath}
  (\mu, z)\in E_1 \mapsto
 %\left\{    \begin{array}[c]{l}
     \int_0^T \int_{\T^d} \tilde L (x,\mu(t,x),z(t,x) ) dx dt. % \quad \hbox{if } \mu \ge 0,
%    \end{array} \right.
\end{displaymath}
Therefore, the map  $ (\phi,\gamma)\in \cK  \mapsto \int_0^T \int_{\T^d} K(x,  \gamma^{\rm ac}(t,x), D\phi(t,x)) dxdt$ is upper semi-continuous for the weak * topology of $ L^\beta(0,T;W^{1,\beta}(\T^d))\times \cM$, and therefore,
\begin{equation}
  \label{eq:59}
\int_0^T \int_{\T^d} K(x,  \gamma^{\rm ac}(t,x), D\phi(t,x)) dxdt \ge \limsup_{n\to \infty}  \int_0^T \int_{\T^d} K(x,  \gamma_n(t,x), D\phi(t,x)) dxdt.
\end{equation}
\paragraph{Step 3}
Let $\zeta_n = \int_{\T^d } m_0(x)\phi_n(t,x)dx$ and  $\zeta= \int_{\T^d } m_0(x)\phi(t,x)dx$. We have seen that  $\zeta_n\rightharpoonup \zeta$ in $L^\infty((0,T))$  weak *; from the inequation satisfied by $\phi_n$,  
\begin{displaymath}
  \frac {d \zeta_n}{dt} (t)- \nu\int_{T^d  } Dm_0(x) D\phi_n(t,x)dx  -\int_{\T^d }  (\gamma_n^1 (t,x) +\gamma_n^2 (t,x) ) m_0(x) dx \ge 0,
\end{displaymath}
which implies that 
\begin{displaymath}
  \frac {d \zeta_n}{dt} (t)- \nu\int_{T^d  } Dm_0(x) D\phi_n(t,x)dx  -\int_{\T^d }  \gamma_n^1 (t,x)  m_0(x) dx \ge 0.
\end{displaymath}
Thanks to the a priori bounds on $\phi_n$ and $\gamma_n ^1$, this implies that 
\begin{displaymath}
  \zeta_n(t)\ge \zeta_n(0)- C  ( t^{\frac 1 {\beta ^*} } + t^{\frac 1 {q}}),\quad \forall t\in [0,T].
\end{displaymath}
Letting $n$ tend to $+\infty$, 
\begin{displaymath}
  \zeta(t)\ge \limsup_{n\to \infty} \zeta_n(0)- C  ( t^{\frac 1 {\beta ^*} } + t^{\frac 1 {q}}), \quad \hbox{a.e. } t\in [0,T].
\end{displaymath}
Hence $ \zeta(0)\ge \limsup_{n\to \infty} \zeta_n(0)$, i.e. $ \int_{\T^d } m_0(x)\phi(t,x)dx\ge \limsup_{n\to \infty}\int_{\T^d } m_0(x)\phi_n(t,x)dx$.
\paragraph{Step 4} By combining the results of steps 2 and 3, we see that 
\begin{displaymath}
  J(\phi,\gamma)\ge \limsup_{n\to \infty} \cA(\phi_n)= \sup_{\tilde \phi\in \cK_0 }\cA(\tilde \phi)=\sup_{ (\tilde \phi, \tilde \gamma) \in \cK }J(\tilde \phi,\tilde \gamma),
\end{displaymath}
Using the fact that $\gamma^1\le -\ell(\cdot,0)$ a.e. and  (\ref{eq:30}), we can always decrease $\gamma^2$  where $D\phi=0$ in such a way that $ (\phi,\gamma)\in \widetilde \cK$ and the value of $J(\phi,\gamma)$ is preserved.
\end{proof}

\section{Weak solutions to the system of PDEs }
\label{sec:weak-solut-syst}
\begin{definition}
  \label{sec:weak-solut-syst-1}
A pair $(\phi,m)\in L^\beta(0,T;W^{1,\beta}(\T^d))\times L^q((0,T)\times \T ^d)$ with $m\ge 0$ almost everywhere,   is a weak solution of (\ref{eq:1})-(\ref{eq:3}) if
\begin{enumerate}
\item   $  m^{-\alpha } | D\phi|^\beta   1_{m>0} \in L^1((0,T)\times \T ^d)$, $D\phi(t,x)=0$ a.e. in the region $\{(t,x): m(t,x)=0\}$,
   $  m^{1-\alpha } | D\phi|^\beta  \in L^1((0,T)\times \T ^d)$ and  % $ mH_p(x, m, D\phi) \in L^1((0,T)\times \T ^d)$ and
  \[\int_0^T \int_{\T^d} m(t,x) L \left(x, m( t,x), H_p(x, m(t,x), D\phi(t,x))   \right ) dxdt <+\infty.\]
\item The following: 
  \begin{equation}
    \label{eq:60}
    \begin{split}
&    \frac{\partial \phi} {\partial t} (t,x)  + \nu \Delta \phi(t,x)  +  H
    (x, m(t,x), D \phi(t,x) )     + m(t,x) H_m(x, m(t,x), D \phi(t,x) ) \ge 0,\\ &\phi(T,x)\le u_T(x),      
    \end{split}
  \end{equation}
  holds in the sense of distributions, with the convention that 
  \begin{equation}
    \label{eq:61}
 H(x, m, p )     + m H_m(x, m,p )=\ell(x,0) \quad \hbox{if} \quad  m=0 \hbox{ and }p=0.
  \end{equation}
\item The following: 
  \begin{equation}
    \label{eq:62}
    \begin{split}
    &      \ds \frac{\partial m} {\partial t} (t,x)  - \nu \Delta  m(t,x) +  \diver   (m(t,\cdot)H_p(\cdot, m(t,\cdot), D\phi(t,\cdot)) )(x)  =0 ,
\\ &    m(0,x)= m_0(x) 
    \end{split}
\end{equation}
  holds in the sense of distributions
\item There exists a constant $C$ such that for each function $\xi\in Lip(\T^d)$,  \[ \|\int_{\T^d} \xi(x)\phi(t,x)dx\|_{L^\infty(0,T)}\le C \|\xi\|_{L^\infty (\T^d)}.\]
\item The following identity holds 
  \begin{equation}
    \label{eq:63}
    \begin{split}
      0=& \ds \int_{0}^T\int_{\T^d}  m(t,x) L (x, m(t,x),  H_p(x, m(t,x), D\phi(t,x)))dx dt
      +\int_{0}^T \int_{\T^d} m^2 (t,x) H_m(x, m(t,x), D\phi(t,x)) dx dt 
      \\ &\ds +\int_{\T^d} m(T,x)   u_T( x)  dx -\int_{\T^d} m_0(x)   \phi(0^+, x)  dx
    \end{split}
  \end{equation}
\end{enumerate}
\end{definition}

\begin{theorem}
  \label{sec:weak-solut-syst-2}
Let $(m^*, z^*)\in \cK_1$ be the minimizer of (\ref{eq:23}) and  $(\phi^*, \gamma^*)\in \widetilde \cK$ be a maximizer of (\ref{eq:45}). 
%such that $\gamma ^{*,\rm ac}(t,x) \le -H(0,x,D \phi^*(t,x))$ a.e.. 
Then $(\phi^*, m^*)$ is a  weak solution of (\ref{eq:1})-(\ref{eq:3}), and $ z^*(t,x)= m^*(t,x)H_p(x, m^*(t,x),D \phi^*(t,x) )$, 
$\gamma ^{*, {\rm ac}}(t,x) \ge -H(x,  m^*(t,x), D \phi^*(t,x)) - m^*(t,x)H_m(x,  m^*(t,x), D \phi^*(t,x))$ almost everywhere.\\
Conversely, any weak solution $(\bar \phi,\bar m)$ of (\ref{eq:1})-(\ref{eq:3}) is such that the pair $(\bar m,   \bar m H_p(x, \bar m,D \bar \phi ) )$ is the minimizer of (\ref{eq:23}) and $(\bar \phi, \bar \gamma)$ is a maximizer of (\ref{eq:45}), with \[ \bar \gamma(t,x)= -H(x, \bar m(t,x), D\bar \phi(t,x)) - \bar m(t,x)H_m(x, \bar m(t,x), D\bar \phi(t,x)),\]
with the convention (\ref{eq:61}).
\end{theorem}

\begin{proof}
Let $(m^*, z^*)\in \cK_1$ be the minimizer of (\ref{eq:23}), (which implies that  $m^*\in L^q((0,T)\times \T^d)$ and that $ \int_0^T \int_{\T^d} \widetilde L (x,m^*(t,x), z^*(t,x)) dx dt <+\infty$) and  $(\phi^*, \gamma^*)\in \widetilde \cK$ be a maximizer of (\ref{eq:45}). 
%such that $\gamma ^{*,\rm ac} \le -H(0,x,D \phi^*)$ a.e..
We know that $ \int_0^T \int_{\T^d}  K(x, \gamma ^{*,\rm ac}(t,x), D \phi^* (t,x)) dx dt >-\infty$. Hence,
\begin{displaymath}
  \begin{split}
&     \int_0^T \int_{\T^d} \widetilde L (x,m^*(t,x), z^*(t,x)) dx dt  + \int_{\T^d}     m^*(T,x) u_T(x) dx \\
=
&\int_0^T \int_{\T^d} K(x, \gamma^{*,{\rm ac}}(t,x), D\phi^*(t,x)) dxdt +\int_{\T^d} m_0(x) \phi^*(0,x) dx \\
\le   &\int_0^T \int_{\T^d}   m^* (t, x)  \left(  \gamma^{*,{\rm ac}}(t,x)   +  H(x, m^*(t,x), D\phi^*(t,x))\right) dxdt +\int_{\T^d} m_0(x) \phi^*(0,x) dx
  \end{split}
\end{displaymath}
But from Lemma \ref{sec:relaxed-problem-3} and especially (\ref{eq:49})-(\ref{eq:50}), we see that the latter inequality is in fact an equality; then,
 from the last part of  Lemma \ref{sec:relaxed-problem-3}, this yields that $z^*(t,x)=m^*(t,x)H_p(x, m^*(t,x),D\phi^*(t,x))$
and that 
\begin{displaymath}
  \widetilde L (x,m^*(t,x), z^*(t,x))=\widetilde L \left(x, m^*( t,x), m^*(t,x) H_p(x, m^*(t,x), D\phi^*(t,x)\right ) 
\end{displaymath}
 almost everywhere.
We have proved that $m^*H_p(\cdot, m^*,D\phi^*)\in  L^1((0,T)\times \T ^d)$ and that (\ref{eq:62}) holds.
\\
Moreover
\begin{displaymath}
\int_0^T \int_{\T^d} K(x, \gamma^{*,{\rm ac}}(t,x), D\phi^*(t,x)) dxdt
=   \int_0^T \int_{\T^d}   m^* (t, x)  \left(  \gamma^{*,{\rm ac}}(t,x)   +  H(x, m^*(t,x), D\phi^*(t,x))\right) dxdt.
\end{displaymath}
This implies that 
\begin{equation}
  \label{eq:64}
   K(x, \gamma^{*,{\rm ac}}(t,x), D\phi^*(t,x))=  m^* (t, x)  \left(  \gamma^{*,{\rm ac}}(t,x)   +  H(x, m^*(t,x), D\phi^*(t,x))\right)
\end{equation}
almost everywhere. In view of Remark~\ref{sec:two-optim-probl-4}, this implies that $D\phi^*(t,x)=0$ almost everywhere in the region where
$m^* (t, x) =0$, and that 
\begin{displaymath}
 \gamma^{*,{\rm ac}}(t,x)=  -H(x, m^*(t,x), D\phi^*(t,x)) - m^*(t,x)H_m(x, m^*(t,x), D\phi^*(t,x)), \quad \quad \hbox{a.e. where }    m ^*(t,x)>0
\end{displaymath}
and with the convention (\ref{eq:61}), we see that
    \begin{displaymath}
 \gamma^{*,{\rm ac}}(t,x)  \ge   -H(x, m^*(t,x), D\phi^*(t,x)) - m^*(t,x)H_m(x, m^*(t,x), D\phi^*(t,x)), \quad \quad \hbox{a.e.} .         \end{displaymath}
These observation imply that $  (m^*)^{-\alpha } | D\phi^*|^\beta  1_{m^* >0} \in L^1((0,T)\times \T ^d)$.
Using also (\ref{eq:64}), we see that $  (m^*)^{1-\alpha } | D\phi^*|^\beta  \in L^1((0,T)\times \T ^d)$.
From the inequalities satisfied by $\gamma^{*,{\rm ac}}$, we see  that (\ref{eq:60}) holds.
\\
Finally, (\ref{eq:63}) is obtained by combining all the points above. We have proved that $(\phi^*, m^*)$ is a weak solution in the sense of Definition~\ref{sec:weak-solut-syst-1}.\\
\\

Suppose now that $(\bar \phi, \bar m)$ is a weak solution in the sense of Definition~\ref{sec:weak-solut-syst-1}. 
Let us choose $\bar \gamma(t,x)= -H(x, \bar m(t,x), D\bar \phi(t,x)) - \bar m(t,x)H_m(x, \bar m(t,x), D\bar \phi(t,x))$
always with the convention (\ref{eq:61}),
(we split $\bar \gamma$ as follows: $\bar \gamma= \bar \gamma^1+ \bar \gamma^2$,
where $\bar \gamma^1 = -\ell(\cdot, \bar m)-\bar m \frac {\partial \ell}{\partial m} (\cdot, \bar m)$,
$ \bar \gamma^2=  (1-\alpha ) \frac {|D\bar \phi|^\beta}{\bar m^\alpha} 1_{\{\bar m>0\}}$)
 and  
$\bar z(t,x)=\bar m(t,x)H_p(x, \bar m(t,x),D\bar \phi(t,x))$.
It is clear that $(\bar \phi, \bar \gamma)\in \widetilde \cK$ %  with $\bar \gamma \le -H(0,x,D \bar \phi)$ a.e.,
 and that $( \bar m,  \bar z)\in \cK_1$.\\

From the definition of $\bar \gamma$, we see that a.e., 
$\bar m(t,x) ={\rm argmin}_{\mu\ge 0} \left( \mu \bar \gamma(t,x) +\mu  H(x, \mu, D\bar \phi(t,x)) \right)$. Moreover, since a.e.,
$  H_p(x, \bar m(t,x),D\bar \phi(t,x)) =  {\rm argmin}_{\xi}    \left( \xi\cdot D\bar \phi(t,x)+ L(x, \bar m(t,x) , \xi) \right) $, we see that a.e. 
$(\bar m(t,x), \bar z(t,x) ) =  {\rm argmin}_{(\mu,z) }   \left( \mu \bar \gamma(t,x) +  z\cdot D\bar \phi(t,x)+ \widetilde L(x, \mu , z)\right) $, i.e.
\begin{equation}
  \label{eq:65}
K(x, \bar \gamma(t,x),  D\bar \phi(t,x))=  \bar m(t,x) \bar \gamma(t,x) +   \bar z(t,x)\cdot D\bar \phi(t,x)+ \widetilde L(x, \bar m(t,x) ,  \bar z(t,x)).
\end{equation}
Note also that $0\le      \int_0^T \int_{\T^d}  \bar m(t,x)\bar \gamma^2(t,x)   dx dt  <+\infty$ and that
\begin{displaymath}
 \int_0^T \int_{\T^d}  |\bar z(t,x)\cdot D\bar \phi(t,x)| dx dt  \le C   \int_0^T \int_{\T^d} \bar m ^{1-\alpha}(t,x)  |D \bar \phi(t,x)|^\beta  dx dt<+\infty.
\end{displaymath}

  Let $(m^* , z^*)$ be the solution of (\ref{eq:23}). 
From Lemma \ref{sec:relaxed-problem-3},
\begin{displaymath}
  \begin{split}
      \cB(m^*, z^*)=&   \int_0^T \int_{\T^d} \widetilde L(x, m^*(t,x), z^* (t,x)) dxdt +\int_{\T^d} m^* (T,x) u_T (x) dx   \\
  \ge&  \int_0^T \int_{\T^d}    m^*(t,x) \left( \bar \gamma (t,x) + H(x,   m^*(t,x), D\bar \phi (t,x) ) \right) dx dt+ \int_{\T^d}  m_0(x) \bar \phi (0,x) dx   .
\end{split}
\end{displaymath}
Hence, 
\begin{displaymath} \begin{split}
      \cB(m^*, z^*)\ge  & \int_0^T \int_{\T^d}   K (x,  \bar \gamma (t,x),  D\bar \phi (t,x) ) dx dt + \int_{\T^d}  m_0(x) \bar \phi (0,x) dx\\
 =&  \int_0^T \int_{\T^d}    \bar m(t,x) \left( \bar \gamma (t,x) + H(x,   \bar m(t,x), D\bar \phi (t,x) ) \right) dx dt+ \int_{\T^d}  m_0(x) \bar \phi (0,x) dx,
\end{split}
\end{displaymath}
where the first line comes from (\ref{eq:65}) and the definition of $K$,
 and the last line comes from the definition of $\bar \gamma$. \\
 Finally, from the latter inequality and (\ref{eq:63}), we deduce that
 \begin{displaymath} \begin{split}
      \cB(m^*, z^*)\ge&  \int_0^T \int_{\T^d} \widetilde L(x, \bar m(t,x), \bar z (t,x)) dxdt +\int_{\T^d} \bar m (T,x) u_T (x) dx \\
= & \cB(\bar m, \bar z).
\end{split}
\end{displaymath}
Therefore $(\bar m, \bar z)$ achieves the minimum in (\ref{eq:23}) and $\bar m=m^*$, $\bar z=z^*$.
 \\
It remains to prove that $(\bar \phi,\bar \gamma)$ achieves the maximum in (\ref{eq:45}).  We deduce from Lemma \ref{sec:two-optim-probl-2}, Proposition \ref{sec:relaxed-problem-2} and the latter point that 
\begin{displaymath}
  \begin{split}
  \max_{(\phi,\gamma)\in \cK } J(\phi,\gamma)= \cB(\bar m, \bar z)&=\int_0^T \int_{\T^d} \widetilde L(x, \bar m(t,x), \bar z (t,x)) dxdt +\int_{\T^d} \bar m (T,x) u_T (x) dx \\
 &=  - \int_0^T \int_{\T^d}   \bar m ^2(t,x)   H_m(x, \bar m(t,x), D \bar\phi (t,x)) dxdt +\int_{\T^d} \bar m_0 (x) \bar \phi(0^+, x)  dx ,
  \end{split}
\end{displaymath}
where the last line comes from (\ref{eq:63}) and the definition of $\bar z$. But, using the definition of $\bar \gamma$,
\begin{displaymath}
  \begin{split}
  - \bar m ^2(t,x)   H_m(x, \bar m(t,x), D \bar\phi (t,x))=& \bar m(t,x) \bar \gamma (t,x) +\bar m(t,x)  H(x, \bar m(t,x) ,D \bar\phi (t,x) )      \\
= & \bar m(t,x) \bar \gamma (t,x) +    \bar z(t,x)\cdot D \bar\phi (t,x)+   \widetilde L (x, \bar m(t,x) , \bar z (t,x) )       \\
=& K(x, \bar \gamma(t,x), D\bar\phi(t,x) )
  \end{split}
\end{displaymath}
where the second line is obtained using the definition of $\bar z$ and the third line is obtained using (\ref{eq:65}).
Combining the latter two observations, we see that
\begin{displaymath}
  \max_{(\phi,\gamma)\in \cK } J(\phi,\gamma)= \int_0^T \int_{\T^d}K(x, \bar \gamma(t,x), D\bar\phi(t,x) )dxdt +\int_{\T^d} \bar m_0 (x) \bar \phi(0^+, x)  dx = J(\bar \phi,\bar \gamma),
\end{displaymath}
which concludes the proof.
\end{proof}

\begin{theorem}
  \label{sec:weak-solut-syst-3}
There exists a unique  weak solution of (\ref{eq:1})-(\ref{eq:3}).
\end{theorem}

 \begin{proof}
   Existence is a direct consequence of Theorem \ref{sec:weak-solut-syst-2}, Proposition \ref{sec:exist-solut-relax-1} and Lemma \ref{sec:two-optim-probl-2}. \\
 From Theorem \ref{sec:weak-solut-syst-2}, we also see that   any weak solution $(\bar \phi,\bar m)$ of (\ref{eq:1})-(\ref{eq:3}) is such that the pair $(\bar m,   \bar m H_p(x, \bar m,D \bar \phi ) )$ is the minimizer of (\ref{eq:23}). Thus $( \bar m, \bar m H_p(x, \bar m,D \bar \phi ))$ is unique.
\\
Consider now two weak solutions $(\bar\phi_1, \bar m)$ and $(\bar\phi_2, \bar m)$ of  (\ref{eq:1})-(\ref{eq:3}). 
We know that $\bar m H_p(x, \bar m,D \bar \phi_1 )=\bar m H_p(x, \bar m,D \bar \phi_2 )$: this implies that 
$ |D\bar \phi_1(t,x)|^{\beta-2 }D\bar \phi_1(t,x)= |D\bar \phi_2(t,x)|^{\beta-2} D\bar \phi_2(t,x)$ at almost every $(t,x)$ such that $\bar m (t,x)>0$:
hence at almost every $(t,x)$ such that $\bar m (t,x)>0$,   $ |D\bar \phi_1(t,x)|= |D\bar \phi_2(t,x)|$ and finally $ D\bar \phi_1(t,x)=  D\bar \phi_2(t,x)$.
Moreover, $D\bar \phi_1(t,x)=D\bar \phi_2(t,x)=0$ at almost every $(t,x)$ such that $\bar m(t,x)=0$.
%because both $\bar m^{-\alpha} |D\bar \phi_1|^\beta$ and $\bar m ^{-\alpha}|D\bar \phi_2|^\beta$ are integrable. 
 Therefore,  $D\bar \phi_1(t,x)=D\bar \phi_2(t,x)$ at almost every $(t,x)$.
This means that $\bar \phi_1 -\bar \phi_2$ only depends on $t$, and that $\bar \gamma_1=\bar \gamma_2$ a.e., where  $\bar \gamma_i(t,x)= -H(x, \bar m(t,x), D\bar \phi_i(t,x)) - \bar m(t,x)H_m(x, \bar m(t,x), D\bar \phi_i(t,x))$, using convention (\ref{eq:61}).
\\
 Using (\ref{eq:63}), we deduce from the previous points that $\bar \phi_1(t=0^+,\cdot)=\bar \phi_2(t=0^+,\cdot)$ a.e. in $\T^d$.
\\
We set $\bar \gamma=  \bar \gamma_1=\bar \gamma_2\in L^1 ((0,T)\times \T^d)$  and $     \bar z= \bar m H_p (x, \bar m, D\bar \phi_1)=\bar m H_p (x, \bar m, D\bar \phi_2)$. Going back to the proof of Lemma \ref{sec:relaxed-problem-3}, we see that for almost all $t$, 
both $(\bar \phi_1, \bar \gamma)$ and $(\bar \phi_2, \bar \gamma)$ achieve the equality in (\ref{eq:50}) with $(m,z)=(\bar m, \bar z)$.
From the previous points, this implies that
\begin{displaymath}
    \int_{\T^d } \bar m(t, x) \bar \phi_1 (t, x) dx=     \int_{\T^d } \bar m(t, x) \bar \phi_2 (t, x) dx
\end{displaymath}
for almost every $t$. Since $\bar \phi_1 (t, x)-\bar \phi_2 (t, x)$ only depends on $t$ and $\bar m(t)$ is a probability measure, 
 the latter implies that $\phi_1=\phi_2$ holds almost everywhere.
\end{proof}

\section*{\bf Acknowledgements}
 The first author  was partially funded  by the ANR projects ANR-12-MONU-0013 and ANR-12-BS01-0008-01.

{\small
\bibliographystyle{amsplain}
\bibliography{wmftc}

\def\cprime{$'$} \def\cprime{$'$}
\providecommand{\bysame}{\leavevmode\hbox to3em{\hrulefill}\thinspace}
\providecommand{\MR}{\relax\ifhmode\unskip\space\fi MR }
% \MRhref is called by the amsart/book/proc definition of \MR.
\providecommand{\MRhref}[2]{%
  \href{http://www.ams.org/mathscinet-getitem?mr=#1}{#2}
}
\providecommand{\href}[2]{#2}
\begin{thebibliography}{10}

\bibitem{MR3135339}
Y.~Achdou, \emph{Finite difference methods for mean field games},
  Hamilton-{J}acobi equations: approximations, numerical analysis and
  applications (P.~Loreti and N.~A. Tchou, eds.), Lecture Notes in Math., vol.
  2074, Springer, Heidelberg, 2013, pp.~1--47.

\bibitem{YAML}
Y.~Achdou and M.~Lauri{\`e}re, \emph{On the system of partial differential
  equations arising in mean field type control}, DCDS A (September 2015).

\bibitem{MR3037035}
A.~Bensoussan and J.~Frehse, \emph{Control and {N}ash games with mean field
  effect}, Chin. Ann. Math. Ser. B \textbf{34} (2013), no.~2, 161--192.

\bibitem{MR3134900}
A.~Bensoussan, J.~Frehse, and P.~Yam, \emph{Mean field games and mean field
  type control theory}, Springer Briefs in Mathematics, Springer, New York,
  2013.

\bibitem{MR3116016}
P.~Cardaliaguet, G.~Carlier, and B.~Nazaret, \emph{Geodesics for a class of
  distances in the space of probability measures}, Calc. Var. Partial
  Differential Equations \textbf{48} (2013), no.~3-4, 395--420.

\bibitem{cardaliaguet2014second}
P.~Cardaliaguet, J~Graber, A.~Porretta, and D.~Tonon, \emph{Second order mean
  field games with degenerate diffusion and local coupling}, arXiv preprint
  arXiv:1407.7024 (2014).

\bibitem{MR3091726}
R.~Carmona and F.~Delarue, \emph{Mean field forward-backward stochastic
  differential equations}, Electron. Commun. Probab. \textbf{18} (2013), no.
  68, 15.

\bibitem{MR3045029}
R.~Carmona, F.~Delarue, and A.~Lachapelle, \emph{Control of {M}c{K}ean-{V}lasov
  dynamics versus mean field games}, Math. Financ. Econ. \textbf{7} (2013),
  no.~2, 131--166.

\bibitem{MR1022305}
R.~J. DiPerna and P.-L. Lions, \emph{Ordinary differential equations, transport
  theory and {S}obolev spaces}, Invent. Math. \textbf{98} (1989), no.~3,
  511--547.

\bibitem{gomes2014existence}
D.~A Gomes and H.~Mitake, \emph{Existence for stationary mean field games with
  quadratic hamiltonians with congestion}, arXiv preprint arXiv:1407.8267
  (2014).

\bibitem{MR3195844}
D.~A. Gomes and J.~Sa{\'u}de, \emph{Mean field games models---a brief survey},
  Dyn. Games Appl. \textbf{4} (2014), no.~2, 110--154.

\bibitem{MR2269875}
J-M. Lasry and P-L. Lions, \emph{Jeux \`a champ moyen. {I}. {L}e cas
  stationnaire}, C. R. Math. Acad. Sci. Paris \textbf{343} (2006), no.~9,
  619--625.

\bibitem{MR2271747}
\bysame, \emph{Jeux \`a champ moyen. {II}. {H}orizon fini et contr\^ole
  optimal}, C. R. Math. Acad. Sci. Paris \textbf{343} (2006), no.~10, 679--684.

\bibitem{MR2295621}
\bysame, \emph{Mean field games}, Jpn. J. Math. \textbf{2} (2007), no.~1,
  229--260.

\bibitem{PLL}
P-L. Lions, \emph{Cours du {C}oll{\`e}ge de {F}rance},
  http://www.college-de-france.fr/default/EN/all/equ$_-$der/, 2007-2011.

\bibitem{MR0221595}
H.~P. McKean, Jr., \emph{A class of {M}arkov processes associated with
  nonlinear parabolic equations}, Proc. Nat. Acad. Sci. U.S.A. \textbf{56}
  (1966), 1907--1911.

\bibitem{porretta2014}
A.~Porretta, \emph{Weak solutions to {F}okker-{P}lanck equations and mean field
  games}, Archive for Rational Mechanics and Analysis (2014), 1--62 (English).

\bibitem{MR0310612}
R.~T. Rockafellar, \emph{Integrals which are convex functionals. {II}}, Pacific
  J. Math. \textbf{39} (1971), 439--469.

\bibitem{MR1451876}
\bysame, \emph{Convex analysis}, Princeton Landmarks in Mathematics, Princeton
  University Press, Princeton, NJ, 1997, Reprint of the 1970 original,
  Princeton Paperbacks.

\bibitem{MR1108185}
A-S. Sznitman, \emph{Topics in propagation of chaos}, \'{E}cole d'\'{E}t\'e de
  {P}robabilit\'es de {S}aint-{F}lour {XIX}---1989, Lecture Notes in Math.,
  vol. 1464, Springer, Berlin, 1991, pp.~165--251.

\end{thebibliography}
}

\end{document}